\documentclass[10pt, 3p, times, reqno, a4paper]{amsart}
\usepackage[right=2 cm,left=2 cm,top=1.7 cm,bottom=1.7 cm]{geometry}
\usepackage[utf8]{inputenc}
\usepackage{amsmath,amsxtra,amssymb,latexsym,amscd,amsthm,amsfonts,amstext,upgreek,dsfont,eufrak,cite,amsbsy,amssymb,mathrsfs,float}
\usepackage{color}
\newcommand {\norm}[1] {\left\| #1 \right\|}

\newcommand {\Bnorm}[1] {\Big\| #1 \Big\|}

\newcommand{\parentheses}[1]{\left(#1\right)}\newcommand{\sqrbrackets}[1]{\left[#1\right]}

\def\dis{\displaystyle}
\def\N{\mathbb{N}}

\def\R{\mathbb{R}}

\def\sumuse{\sum_{j=1}^{\infty}}
\def\.{\cdot}
\def\Lp{L^p(\R^N)}
\def\Lq{L^q(\R^N)}

\def\Lh{L^h(\R^N)}
\def\Lep{L^{\Xi}(\R^N)}
\def\Lpx{L^p_{x}}

\def\Linftyx{L^{\infty}_{x}}

\def\0Lep{L^{\Xi}_{0}(\R^N)}
\def\Linfty{L^{\infty}(\R^N)}
\def\p{u}

\newcommand{\opnorm}[1] {\left|\mkern-1.5mu\left|\mkern-1.5mu\left| #1 \right|\mkern-1.5mu\right|\mkern-1.5mu\right|}

\usepackage{hyperref}
\usepackage{verbatim}

\usepackage{multicol,graphicx}
\usepackage{array,multirow}
\usepackage{enumerate,booktabs}
\usepackage{stackengine}

\allowdisplaybreaks
\newcommand{\vertiii}[1]{{\left\vert\kern-0.15ex\left\vert\kern-0.15ex\left\vert #1 \right\vert\kern-0.15ex\right\vert\kern-0.15ex\right\vert}}

\newtheorem{theorem}{{\bf Theorem}}[section]
\theoremstyle{definition} \newtheorem{definition}[theorem]{\bf Definition}

\theoremstyle{plain} \newtheorem{lemma}[theorem]{Lemma}
\newtheorem{proposition}{Proposition}[section]

\newtheorem{remark}{Remark}[section]

\newcommand{\al}{\alpha}

\newcommand{\ep}{\epsilon}

\def\nn{\nonumber}

\def\bq{\begin{equation}}
\def\eq{\end{equation}}
\def\bqq{\begin{eqnarray*}}
	\def\eqq{\end{eqnarray*}}

\usepackage{extarrows}

\title[Well-posed results for time fractional biharmonic equation]{On the initial value problem for a class of nonlinear biharmonic equation  with time-fractional derivative}

\author[N.A. Tuan]{Anh Tuan Nguyen}
\address[N.A. Tuan]{Division of Applied Mathematics,  Thu Dau Mot University, Binh Duong Province, Vietnam}
\email{nguyenanhtuan@tdmu.edu.vn}

\author[T. Caraballo]{Tomás Caraballo}
\address[T. Caraballo]{Departamento de Ecuaciones Diferenciales y Análisis Numérico C/ Tarfia s/n, Facultad de Matemáticas, Universidad de Sevilla, Sevilla 41080, Spain}
\email{caraball@us.es}

\author[N.H. Tuan]{Nguyen Huy Tuan$^*$}
\address[N.H. Tuan]{Department of Mathematics and Computer Science, University of Science, Ho Chi Minh City, Vietnam; and Vietnam National University, Ho Chi Minh City, Vietnam}
\email{nhtuan@hcmus.edu.vn}

\thanks{$^*$Corresponding author: Nguyen Huy Tuan (nhtuan@hcmus.edu.vn).}

\begin{document}
\mdseries 	
\begin{abstract}
In this work, we investigate the IVP  for a time-fractional fourth-order equation with nonlinear source terms. More specifically, we consider the time-fractional biharmonic with exponential nonlinearity and the time-fractional Cahn-Hilliard equation. By using the Fourier transform concept, the generalized formula for the mild solution as well as the smoothing effects of resolvent operators are proved. For the IVP associated with the first one, by using the Orlicz space with the function $ \Xi(z)=e^{|z|^p}-1$ and some embeddings between it and the usual Lebesgue spaces, we prove that the solution is a global-in-time solution or it shall blow up in a finite time if the initial value is regular. In the case of singular initial data, the local-in-time/global-in-time existence and uniqueness are derived. Also, the regularity of the mild solution is investigated. For the IVP associated with the second one, some modifications to the generalized formula are made to deal with the nonlinear term. We also establish some important estimates for the derivatives of resolvent operators, they are the basis for using the Picard sequence to prove the local-in-time existence of the solution.
		
\vspace*{0.1cm}
		
	\end{abstract}

	\maketitle

\section{Introduction}
Our objective in this paper is to study the following initial value problem associated with the time-fractional derivative with biharmonic operator
\begin{align}\label{GeneralPro}\tag{P}
\begin{cases}
\dis \partial^\al_{0|t}  u(t,x)  + \Delta^2 u(t,x)=G(t,x,u),&\quad  \mbox{in}\quad\R^+\times\R^N, \\ 
\dis u(0,x)= u_0(x), & \quad \mbox{in} \quad \R^N,
\end{cases}
\end{align}
where $0<\al<1$ and, for an absolutely continuous in time function $ w $, the definition the Caputo time-fractional derivative operator $\partial^\al_{0|t}$ is introduced in \cite{Rev1} as follows
\begin{equation}\label{Caputo_Def}
\partial^\al_{0|t} w(t)= \frac{1}{\Gamma(1-\al)} \int_0^t (t-s)^{-\al} \frac{dw}{ds} ds,
\end{equation}
here, we assume that the integration makes sense $\Gamma$ is the Gamma function. The main equation in \eqref{GeneralPro} contains the biharmonic operator $\Delta ^2$ which is often called higher-order parabolic equations
and began to receive widespread attention for their surprising and unexpected properties.  More importantly, the higher-order parabolic equations can be used to model many problems in applications, namely, the study of weak interactions of dispersive waves, the theory of combustion, the phase transition, the higher-order diffusion... 
The most common higher-order parabolic equation is probably the polyharmonic heat equation, especially, the fourth-order heat equation or also called the biharmonic heat equation.
The present paper only considers the source function $ G $ to be in two nonlinear cases:  the exponential nonlinear type and the Cahn-Hilliard equation form.

\subsection{{Fractional partial differential equations}}
Over the last decades, the theory of fractional derivatives has been well developed. As a result, the fractional partial differential equations (FPDEs) have also been studied more and more widely. These new kinds of PDEs have many unexpected properties and numerous applications in many applied and theoretical fields of science and engineering. This is why many researchers have shown a big interest in the study of  FPDEs. Recently, there are many interesting works concerning diffusion equations with non-local time derivatives and time-fractional derivatives.  We can refer the reader to some interesting papers on FPDEs, for example, T. Caraballo et al. \cite{Tomas1,Tomas2}, {R. Zacher \cite{Rev6,Rev7}, V. Vespri et al. \cite{Rev2}},  E. Nane et al. \cite{Erkan, Erkan1}, P.N. Carvalho and G. Planas \cite{Planas}, H. Dong and D. Kim \cite{Dong, Dong1}, N.H. Tuan et al. \cite{Tuan, Tuan1} and references therein. {We especially consider the interesting works \cite{Rev3},\cite{Rev4},\cite{Rev5} because our views in approaching the problem are somewhat similar to theirs. Indeed, the works \cite{Rev3},\cite{Rev4} study time-fractional problems with second order differential operators through the fundamental solution. While in \cite{Rev3}, Dipierro and  co-authors  establish existence and uniqueness for the solution in an appropriate functional space,  and in \cite{Rev4}, Zacher et al.  consider decay estimates of the solution. The work \cite{Rev5} investigates the same problem as \cite{Rev4} but the results are provided in a bounded domain of $ \R^N $ and applied to investigate some specific examples corresponding to their kernel $ k $. When considered these works, we found some important remarks about appropriate functional spaces to study the fundamental solution (Remark \ref{Rev suggestion}). Motivated by those, we made some detailed observations for ours. On the other hand, our main contributions are the results for models with the biharmonic operators whose properties are somewhat different from the second order ones. Furthermore, we focus on studying specific effects of different types of nonlinearity to our mild solution. To provide a clearer view, we will present specific discussions below.}
\subsection{Discussion on Problem \eqref{GeneralPro} with exponential nonlinearity }
The source function $ G $ is the exponential nonlinearity satisfying $ G(0)=0 $ and
\begin{align}\label{Property of J}
\left|G(u)-G(v)\right|&\le L|u-v|\parentheses{|u|^{m-1}e^{\kappa|u|^p}+|v|^{m-1}e^{\kappa|v|^p}},
\end{align}
for every $ u,v\in\R, m>2\text{~or~}m=1,p>1 $ and $ L $ is a positive constant independent of $ u,v. $  In the following, we will discuss  in more detail why we chose this function $G$ as in \eqref{Property of J}. 

\noindent$\bullet$ \textit{In terms of mathematical theory}: It was common knowledge that when we consider the IVP for the classical Schr\"odinger equations with the polynomial nonlinearity $ u|u|^{p-1},~p\in(1,\infty) $ and a initial data function in $ H^s(\R^N),~s\in[0,N/2) $, the value $\mathbf{\bar{c}}=1+4(N-2s)^{-1}$ is called the critical exponent. Then, the power case $ p $ of the nonlinear function is equal to (respectively less than) $ \mathbf{\bar{c}} $ is called the critical case (respectively subcritical case). However, when considering the functional space $ H^{\frac{N}{2}}(\R^N) $, the critical value $ \mathbf{\bar{c}} $ will be larger than any power exponent of the polynomial nonlinearity. Hence, the nonlinear functions of exponential type grow higher than any kind of a power nonlinearity at infinity and also vanishes like a power at zero, which can be seen as the critical nonlinearity of this case. This is also one of the reasons why exponential nonlinearity has been studied by many mathematicians not only for both the Schr\"odinger equation and some other types of PDEs. To provide an overview of this kind of nonlinearity, let us recall some related works. The framework introduced above is based on results in \cite{Nakamura-Ozawa}. In this work, Nakamura and Ozawa study the small data global $ H^{N/2}(\R^N)-$solution of the IVP for the Schr\"odinger equations with the exponential nonlinearity. The IVP for heat equations with this type of nonlinearity was considered in \cite{Ioku1}. In \cite{Ioku1}, under the smallness assumption on the initial data in the Orlicz space, Ioku has shown the existence of a global-in-time solution of the semilinear heat equations. Under the smallness condition of initial data, decay estimates and the asymptotic behavior for global-in-time solutions of a semilinear heat equation with the nonlinearity given by $ f(u)=|u|^{4/N}ue^{u^2} $ was investigated in \cite{Kawakami Et al.} by Furioli et al. For more results about the exponential nonlinearity, we refer the reader to Bartolucci et al. \cite{Ponce}, Petropoulou \cite{Petropoulou}, Duong \cite{D.V.Duong}, Ioku \cite{Ioku2} and the references therein.

\noindent $ \bullet $ \textit{In terms of application}: These exponential nonlinearities as in \eqref{Property of J} are not only investigated for the nonlinear Schr\"odinger equations but also for other types of PDEs because of many applications in phenomena modeling. Let us mention two well-known applications in combustion theory as follows. The first one is the IVP for the equation $ u_t-\Delta u=ke^u $, it can be used to model the ignition solid fuel. The second one is the description of the small fuel loss steady-state model by the IVP associated with equation $ -\Delta u=ke^{\frac{u}{1+\varepsilon u}}.$ More applications and details can be found in \cite{Combustion Theory}, Kazdan, Warner \cite{diff geometry} and references therein.

\noindent $ \bullet $ \textit{Contributions, challenges and novelties}:
To the best of our knowledge, FPDEs with nonlinearities of exponential type have not been studied yet.  Our work can be seen as one of the first results in this topic.  Due to the nonlinearity of exponential type, it is not possible to apply $L ^p-L^q$ estimates of some previous works \cite{And, And1} to our current problem.  In contrast to the case $ s>N/2, $ the embedding $ \Linfty\hookleftarrow H^{N/2} $ is not true, and in view of  Trudinger--Moser's inequality, we obtain the embeddings $ H^{N/2}(\R^N)\hookrightarrow\Lep\hookrightarrow\Lq $ for any $ p\le q<\infty, $ where $ \Lep $ is the Orlicz space with the function $ \Xi(z)=e^{|z|^p}-1 $ (see Definition \ref{Main_Orlicz_Space}).  To deal with the exponential type,  we shall use the Orlicz space and the embeddings between it and the usual Lebesgue spaces.  However, since our problem is considered in the whole space $ \R^N $, the embedding $ \Lq\hookrightarrow\Lp $ does not hold anymore when $ q>p $.  For Orlicz spaces in classical derivatives, the use of the standard smoothing effect of the exponential resolution and Taylor expansion operators play important roles. However, they are not available for problems with time-fractional derivatives.    
The appearance of the Mittag-Leffler function and the Gamma function also caused a lot of difficulties in setting up some needed estimates related to the Orlicz space for Problem \eqref{GeneralPro}. 
Fortunately, thanks to the results shown in \cite{Mainardi}, the standard smoothing effect can be achieved with a presentation via the $M$--Wright function of the Mittag-Leffler function. We can also overcome some of the difficulties caused by the Gamma function with some special inequalities. 
Our main results in this section are briefly described as follows 
\begin{itemize}
\item  In the regular case of $ u_0~(u_0\in\Lp\cap {C_0(\R^N)}), $ we can derive that our mild solution blows up at a finite time or the maximal time that ensures the unique existence of the solution is infinity. 
\item  With the assumption of the initial value in the space $ \Lep $, the local-in-time existence of mild solution can be obtained by a fixed point argument without any smallness assumptions on the initial function. Furthermore, under the stronger assumption that $ u_0\in\0Lep $, the global well-posed results for the solution will be established. To achieve this goal, we have to use the techniques introduced by Y. Chen et al. \cite{Atienza}, and weighted spaces to deal with the singular term of the mild solution. 
 \end{itemize}
\subsection{Discussion on  Problem \eqref{GeneralPro} with Cahn-Hilliard source term  }
In this subsection,  we introduce and discuss  Problem \eqref{GeneralPro} with another source  $ G(u)\equiv\Delta F(u) $, here $F$  denotes the derivative of double-well potential; in general, we consider a cubic polynomial like
$
F(u)=u^3-u.
$
For the case $\al=1$, Problem \eqref{GeneralPro}  is
reduced to the standard  Cahn-Hilliard equation.   The Cahn–Hilliard equation was proposed for the first time by J. W. Cahn and J. E. Hilliard \cite{1}, and is one of the most often studied problems
of mathematical physics, 
which describes the process of phase separation of
a binary alloy below the critical temperature.  More recently it has appeared in
nano-technology, in models for stellar dynamics, as well as in the theory of galaxy formation
as a model for the evolution of two components of inter-galactic material (see \cite{Blow}). Let us mention some previous works on the standard Cahn-Hilliard equations with derivatives of integer order.  In \cite{Abel}, the authors considered a Cahn-Hilliard equation which is the conserved gradient flow of a nonlocal total free energy functional.  Bosch and Stoll \cite{Bos} proposed a fractional inpainting model based on a fractional-order vector-valued Cahn--Hilliard equation, L.A. Caffarelli,  N.E. Muler \cite{Ca}.  We can list many classical papers related to the study of Cahn–Hilliard equation, see, e.g., Dlotko \cite{Dlotko}, Temam \cite{Temam}, Akagi \cite{Akagi}, Zelik \cite{Zelik1,Zelik2} and the references therein. 

\noindent $ \bullet $ \textit{Contributions, challenges and novelties}: As we mentioned above, when we consider the FPDEs in $ \R^N $, the relationship between two spaces $ \Lp $ and $ \Lq $ with $ p\neq q $ is not fulfilled. Specifically, one of the greatest challenges, when we investigate the Cahn-Hilliard equations, is to deal with a nonlinearity of the form $ G(u)=\Delta F(u) $. Due to the appearance of the Laplacian, when handling the existence and the uniqueness of the mild solution by the successive approximations method and Young's convolution inequality, the property of $ d/dx\parentheses{f(x)*g(x)} $ needs to be applied to make the second-order derivative of the solution representation operator appear. This novelty in this case is setting up the key estimates as in Lemma \ref{lemqt}. It is worth noting that even when we have the main tools available, it is still not an easy task to prove our desired existence results. Besides, when finding the regularity results, we have to estimate the higher-order derivative of the solution representation operators. 
By learning about the results and techniques of  \cite{Luca} and \cite{Liu}, we found a way to obtain the local-in-time existence and uniqueness result for  Problem \eqref{GeneralPro}  with the Cahn-Hilliard source. The  global-in-time existence of the solution is a difficult topic and will probably be studied in a forthcoming work. The section about the time-fractional Cahn-Hilliard equation in this work includes:
\begin{itemize}
\item To find the solution representation by the Fourier transform and some related properties;
\item To establish some useful linear estimates;
\item To prove the existence, uniqueness and regularity of local-in-time solution by using the Picard sequence method and the smallness assumption for the initial data function.
\end{itemize}

\subsection{The outline} In Section 2, we demonstrate an approach to present the formula of the mild solution and, based on it, we establish some important linear estimates. Also in this section, we introduce some notations and definitions related to the so-called Orlicz space, a generalization of
Lebesgue spaces and some useful embeddings between them.  In Section 3, we investigate Problem \eqref{GeneralPro} with the exponential nonlinearities under two separate assumptions on the initial datum function. In particular, for the first assumption, we show that the mild solution exists on $ [0,\infty) $ or blows up in a finite time. The local existence and the global-in-time well-posedness of the solution will be stated under the second assumption of the initial function. The main results on the problem with the second type of nonlinearity source function will be analyzed in Section 4. In general, by using the smallness assumption on the initial function, we derive the local-in-time existence and uniqueness of the mild solution for the IVP associated with the time-fractional Cahn-Hilliard equation. Furthermore, the regularity result will also be proved.

\section{Preliminaries}
\subsection{Generalized mild solution}
It is well known that for the following IVP involving a classical homogeneous biharmonic equation
\begin{align*}
\begin{cases}
\partial_t \varphi(t,x)+\Delta^2 \varphi(t,x)=0,\qquad&\text{in}\quad\R^+\times\R^N,\\
\varphi(t,x)=\varphi_0(x),&\text{in}\quad\{0\}\times\R^N,
\end{cases}
\end{align*}  
the solution is given by
\begin{align*}
\varphi(t,x)=\sqrbrackets{\mathcal{F}^{-1}\parentheses{e^{-t|\xi|^4}}*\varphi_0}(x)=\sqrbrackets{(2\pi)^{\frac{-\N}{2}}\int_{\R^N}e^{i< \xi,x>-t|\xi|^4}d\xi}*\varphi_0(x),
\end{align*}
where the Fourier transform and its inverse are denoted by $ \mathcal{F},\mathcal{F}^{-1}, $ respectively, and
$
<\xi,x>= \sum_{j=1}^N \xi_j x_j,~(\xi,x)\in\R^{2N}.
$
We recall the following lemma for the kernel $ \mathscr K (t,x)= \mathcal F^{-1} (e^{-t|\xi|^4} ). $ 
\begin{lemma}\label{Pre Lp-Lq estimate}
Suppose that $p\ge1$. Then, for any $ t>0 $ we have
\begin{align*}
\|\mathscr{K}(t )\|_{L^p}\le c_p t^{-\frac{N}{4}(1-\frac{1}{p})},
\qquad\|D^m \mathscr{K}(t )\|_{L^p}\le c_{p,m} t^{-\frac{N}{4}(1-\frac{1}{p})-\frac{m}{4}}. 
\end{align*}
\end{lemma}

In view of the above approach, to find the representation for the mild solution to Problem \eqref{GeneralPro}, we consider the IVP for a homogeneous time-fractional biharmonic equation as follows 
\begin{align}\label{HomogeneousIVP Orlicz}
\begin{cases}
\dis \partial^\al_{0|t}\varphi(t,x)+\Delta^2\varphi(t,x)=0,\qquad &\mbox{in}\quad\R^+\times\R^N,\\
\dis \varphi(0,x)=\varphi_0(x),&\mbox{in} \quad \R^N.
\end{cases}
\end{align}
 Applying the Laplace transform, $ \mathscr{L} $, with respect to the time variable to the first equation of \eqref{HomogeneousIVP Orlicz}, we have
\begin{align*}
z^{\alpha}\mathscr{L}\{\varphi\}(z,x)-z^{\alpha-1}\varphi_0(x)+\Delta^2\mathscr{L}\{\varphi\}(z,x)=0.
\end{align*}
Then, by assuming that $ \varphi_0 $ belongs to some appropriate spaces and using the Fourier transform with respect to the spatial variable, the following equation holds
\begin{align*}
z^{\alpha}\mathcal{F}(\mathscr{L}\{\varphi\})(z,\xi)-z^{\alpha-1}\mathcal{F}(\varphi_0)(\xi)+|\xi|^4\mathcal{F}(\mathscr{L}\{\varphi\})(z,\xi)=0.
\end{align*}
By some simple calculations, one has
\begin{align*}
\mathcal{F}(\mathscr{L}\{\varphi\})(z,\xi)=\frac{z^{\alpha-1}}{z^{\alpha}-|\xi|^4}.
\end{align*}
 We now use the inverse Laplace transform to obtain
\begin{align*}
\mathcal{F}(\varphi)(t,\xi)=\widehat{\varphi}(t,\xi)=E_{\alpha,1}\parentheses{-t^{\alpha}|\xi|^4}\widehat{\varphi_0}(\tau).
\end{align*}
Thanks to the Duhamel principle, the Fourier transform of the solution to Problem \eqref{GeneralPro} is given by
\begin{align*}
\widehat{u}(t,\xi)= E_{\al,1} (-t^{\alpha} |\xi|^4)\widehat{u_0}(\xi)+ \int_0^t (t-s)^{\al-1} E_{\al,\al} (- (t-\tau)^\al  |\xi|^4) \widehat{G(u)} (\tau,\xi) d\tau.
\end{align*}
Where $ E_{\al,1},E_{\al,\al} $ are Mittag-Leffler functions. Using the inverse Fourier transform, we obtain
\begin{align*}
u(t,x)= \sqrbrackets{\mathcal F^{-1} \Big(E_{\al,1} (-t^\al |\xi|^4)\Big)*u_0}(x)+ \int_0^t \mathcal F^{-1} \Big((t-\tau)^{\al-1} E_{\al,\al} (-t^\al |\xi|^4)\Big)(x)*G(u(\tau,x))d\tau,
\end{align*}
where we have used the fact that $ \widehat{f*g}(\tau)=\widehat{f}(\tau)\widehat{g}(\tau) $. For convenience, we denote 
\begin{align} \label{a11}
\mathbb K_{1,\al} (t,x)= \mathcal F^{-1} \Big(E_{\al,1} (-t^\al |\xi|^4)\Big)(x),\quad
\mathbb K_{2,\al}(t,x)= \mathcal F^{-1} \Big(t^{\al-1} E_{\al,\al} (-t^\al |\xi|^4)\Big)(x),
\end{align}
and set operators $ \mathbb{Z}_{i,\al}~(i=1,2) $ as follows 
\begin{align*}
\mathbb Z_{i,\al}(t,x) v(t,x)=  \mathbb K_{i,\al} (t,x) * v(t,x)= \int_{\mathbb R^N} \mathbb K_{i,\al} (t,x-y) v(t,y)dy ,~~i=1,2.
\end{align*}
Then, we rewrite our solution formula in a concise form 
\begin{align}\label{MILD SOLUTION FORMULAE}
u(t,x)= \mathbb Z_{1,\al}(t,x)u_0(x)+ \int_0^t \mathbb Z_{2,\al}(t-\tau,x)G(u(\tau,x))d\tau.
\end{align}
{\begin{remark}\label{Rev suggestion}
It is worth noting that the above approach is similar to the common one used to construct the fundamental solution for time-fractional problems with second order differential operators. Let us provide some remarks from interesting works about functional spaces in which the fundamental kernels are considered. 
\begin{enumerate}[(i)]
\item The work \cite{Rev3} studies an evolution problem with the Caputo derivative of order $ \alpha\in(0,1) $ and the Dirac delta distribution centered at $ x = 0 $ (the initial data function). The solution formula of this problem is given by 
\begin{align*}
u(\xi,t)=\mathcal{F}^{-1}\Big(E_{\al,1}\big[(a-4\pi^2b|\xi|^2)t^{\alpha}\big]\Big),\quad a\ge0,b>0.
\end{align*}    
Then, based on it, the authors have made very interesting comments about functional spaces to which the Fourier transform of the solution belongs. More precisely, they showed that $ \mathcal{F}u(\cdot,t) $ is in $ \Lp $ if and only if $ p\in(N/2,\infty) $. This results implies some different case for functional spaces depending on the dimension $ N $.
\item In \cite[Section 3]{Rev4} the authors proved optimal decay estimates for solutions to a time-fractional diffusion equation. Their solution is as follows
\begin{align*}
u(x,t)=\int_{\R^N}Z(x-y,t)u_0(y)dy,~\text{ where }~\mathcal{F}\{Z\}(\xi,t)=E_{\alpha,1}\big(-|\xi|^2t^\al\big).
\end{align*}
From the above, they deduced as a conclusion that $ Z(t) $ fails to belong to $ \Lp $ for $ N\ge4 $ and $ p\ge\frac{N}{N-2} $.
\end{enumerate}
These facts are important when using the fundamental solution to establish well-posed results. In the spirit of the above works, we also present some similar comments on the estimates for kernels $ \mathbb{K}_{1,\alpha},\mathbb{K}_{2,\alpha} $ at Remark \ref{Functional spaces remarks}.
\end{remark}}
In order to achieve the standard smoothing effect of $ \mathbb{Z}_{i,\al}~(i=1,2) $, we also present the mild solution in another form. To this end, we recall the definition of Mittag-Leffler function via the M-Wright type function as follows
\begin{align} \label{e1}
E_{\al,1}(-z)= \int_0^\infty \mathcal M_\al (\zeta) e^{-z\zeta} d\zeta,~~\quad 
E_{\al,\al}(-z)=\int_0^\infty \al\zeta \mathcal M_\al (\zeta) e^{-z\zeta} d\zeta,~~z \in \mathbb C.
\end{align}
Then, we have the second type representation of the solution of Problem \eqref{Main Problem}

\begin{align}\label{Solution Formula}	u(t,x)&=\int_{0}^{\infty}\mathcal{M}_{\alpha}(\zeta)\sqrbrackets{\mathscr K (\zeta t^{\al},x)*\p_0(x)}d\zeta \nn\\
	&+\int_{0}^{t}\int_{0}^{\infty}(t-\tau)^{\alpha-1}\alpha \zeta\mathcal{M}_{\alpha}(\zeta)\sqrbrackets{\mathscr K (\zeta(t-\tau)^{\al},x)*G(u(\tau,x))}d\zeta d\tau.
\end{align}
Due to the great impact of the operator $ \mathbb{K}_{1,\alpha},\mathbb{K}_{2,\alpha} $to our results for mild solutions, we present the following Theorem which can be seen as the combination of Theorem 3.1, Theorem 3.2 and Remark 1.6 of \cite{Wang}.
{\begin{theorem}\label{Continuous_uni_operator_topology}
		Let $ X=\Lp~(1\le p<\infty)$ or $X=C_0(\R^N).$ Then, $ \mathbb{Z}_{1,\al}(t) $ and $ t^{1-\al}\mathbb{Z}_{2,\al}(t) $ are bounded linear operators on $ X $. In addition, for $ w\in X $, $ t\to\mathbb{Z}_{1,\al}(t),~t\to t^{1-\al}\mathbb{Z}_{2,\al}(t) $ are continuous functions from $ \R^+ $ to $ X $. 
\end{theorem}}
\begin{remark}
	{In fact, although theorems of \cite{Wang} can be applied for other spaces in which $ \Delta^2 $ generates a strongly continuous semigroup, in this work, we only focus on the spaces $ \Lp~(1\le p<\infty) $ and $ C_0(\R^N) $.}
\end{remark}
We continue the work by introducing some useful $ L^p- $estimates for the kernel $ \mathbb{K}_{1,\alpha},\mathbb{K}_{2,\alpha} $ by the following lemma.
\begin{lemma} \label{lemqt}
Let {$p\ge1$ and $ k\in\mathbb{N} $} be constants such that {$ k<4-N\parentheses{1-\frac{1}{p}}.$}
Then, there exist two constants $\mathscr C_{k,p}, \overline{ \mathscr C_{k,p}} $ which depend only on $ \al $ and $ N $, such that
\begin{align} \label{K111}
\Big\| D^k \mathbb K_{1,\al} (t ) \Big\|_{L^p} \le \mathscr C_{k,p} (\al, N)  t^{-\frac{\al N}{4}-\frac{\al k }{4}+ \frac{\al N}{4p}} 
\end{align}
and
\begin{align} \label{K22222}
\Big\| D^k \mathbb K_{2,\al} (t ) \Big\|_{L^p} \le \overline{ \mathscr C_{k,p}} (\al, N)    t^{\al-\frac{\al N}{4}-1+\frac{\al N}{4p}-\frac{\al k}{4}}.
\end{align}
\end{lemma}

\begin{proof}~
	
\noindent\textbf{\textit{Step 1. To verify the first inequality}}

In this step, we deal with the term $\mathbb K_{1,\al} (t,x)$. In fact, the  representation  of $\mathbb K_{1,\al} (t,x)$ and  \eqref{e1} together with Fubini's theorem allow us to deduce
\begin{align*}
\mathbb K_{1,\al} (t,x)	= \mathcal F^{-1} \Big(E_{\al,1} (-t^\al |\xi|^4)\Big)(x)&=\parentheses{2\pi}^{-N} \int_{\mathbb R^N}   e^{i <\xi, x>}  \Big(E_{\al,1} (-t^\al |\xi|^4)\Big) d\xi\nn\\
&= \parentheses{2\pi}^{-N}\int_{\mathbb R^N}   e^{i <\xi, x>}  \Big(\int_0^\infty \mathcal M_\al (\zeta) e^{-t^\al \zeta  |\xi|^4} d\zeta \Big) d\xi\nn\\
&=\parentheses{2\pi}^{-N} \int_0^\infty \int_{\mathbb R^N} \mathcal M_\al (\zeta)    e^{i <\xi, x>} e^{-t^\al \zeta  |\xi|^4}    d\xi d\zeta.
\end{align*}
By setting  $\xi = \vartheta  (t^\al \zeta )^{-\frac{1}{4}}$, it is straightforward that  $d\xi = (t^\al \zeta  )^{-\frac{N}{4}} d\vartheta  $ and $|\xi|^4= (t^\al \zeta  )^{-1} |\vartheta|^4$.
Let us  denote by
\begin{align} \label{B}
\overline{\mathscr B}_k(y)= \int_{\mathbb R^N} |\vartheta|^ke^{i <y,\vartheta>  } e^{-|\vartheta|^4} d\vartheta,\quad k\ge0.
\end{align}
By some simple transformations, we find the following equality
\begin{align*}
\mathbb K_{1,\al} (t,x)&=  \int_0^\infty  \int_{\mathbb R^N}   (t^\al \zeta )^{-\frac{N}{4}}  \mathcal M_\al (\zeta)    e^{{i <\vartheta,  x> t^{-\frac{\al}{4}} \zeta ^{-\frac{1}{4}}}}   e^{-|\vartheta|^4} d\vartheta    d\zeta  \nn\\
&= t^{-\frac{\al N}{4}}\Bigg(    \int_0^\infty \zeta^{-\frac{N}{4}} \mathcal M_\al (\zeta)  \int_{\mathbb R^N} e^{{i <\vartheta,  x> t^{-\frac{\al}{4}} \zeta ^{-\frac{1}{4}}}}    e^{-|\vartheta|^4}   d\vartheta    d \zeta    \Bigg)\\
&= t^{-\frac{\al N}{4}}  \int_0^\infty \zeta^{-\frac{N}{4}} \mathcal M_\al (\zeta)   \overline{\mathscr B}_0 (x  (t^\al \zeta  )^{-\frac{1}{4}} )  d\zeta.
\end{align*}
If we set $x  (t^\al \zeta )^{-\frac{1}{4}} =z, $ then it follows immediately  that $dx=  (t^\al \zeta )^{\frac{N}{4}} dz$. Applying Minkowski’s inequality in integral form, we have
{\begin{align*}
		\left(\int_{\R^N}\left|\int_0^\infty \zeta^{-\frac{N}{4}} \mathcal M_\al (\zeta)   \overline{\mathscr B}_0 (x  (t^\al \zeta  )^{-\frac{1}{4}} )  d\zeta\right|^pdx\right)^{\frac{1}{q}}\le\int_0^\infty\left(\int_{\R^N}\left| \zeta^{-\frac{N}{4}} \mathcal M_\al (\zeta)   \overline{\mathscr B}_0 (x  (t^\al \zeta  )^{-\frac{1}{4}} )  \right|^pdx\right)^{\frac{1}{q}}d\zeta.
\end{align*}}
It follows that
\begin{align*}
\Big\|\mathbb K_{1,\al} (t,x)\Big\|_{L^p} 
&\le  t^{-\frac{\al N}{4}} \int_0^\infty    \Bigg( \int_{\mathbb R^N} \Big| \zeta^{-\frac{N}{4}} \mathcal M_\al (\zeta)     \overline{\mathscr B}_0 \left(x  (t^\al \zeta  )^{-\frac{1}{4}} \right) \Big| ^p (t^\al \zeta )^{\frac{N}{4}}  dz \Bigg)^{\frac{1}{p}} d\zeta  \nn\\
&= t^{-\frac{\al N}{4}+\frac{\al N}{4p}}    \left( \int_0^\infty \zeta^{\frac{N}{4p}-\frac{N}{4}}  \mathcal M_\al (\zeta)      d\zeta \right)  \Bigg( \int_{\mathbb R^N}   \Big|  \overline{\mathscr B}_0 (z) \Big|^p dz \Bigg)^{\frac{1}{p}}.
\end{align*}
By setting $  \Theta_{p,N}=\Bigg( \int_{\mathbb R^N}  \Big|  \overline{\mathscr B}_0 (z) \Big|^p dz \Bigg)^{\frac{1}{p}} $ and using Lemma \ref{lem21},  we obtain the following bound
\begin{align} \label{K11}
\Big\|\mathbb K_{1,\al} (t )\Big\|_{L^p} \le   \frac{\Theta_{p,N} \Gamma(\frac{N}{4p}-\frac{N}{4}+1)  }{\Gamma(\frac{\al N}{4p}-\frac{\al N}{4}+1) } t^{-\frac{\al N}{4}+\frac{\al N}{4p}} .
\end{align}
 Next, let us consider the derivative of $\mathbb K_{1,\al} (t,x)$. It is easy to see that
\begin{align*}
\Big\|D \mathbb K_{1,\al} (t )\Big\|_{L^p}= \left( \int_{\mathbb R^N} \big|D \mathbb K_{1,\al} (t,x)\big|^p   dx  \right)^{1/p}.
\end{align*}
We have a view on the modus of the boundness for the  term $D^k \mathbb K_{1,\al} (t,x)$ as follows
\begin{align} \label{d1}
\big|D^k \mathbb K_{1,\al} (t,x)\big| \le 	t^{-\frac{\al N}{4}}  \int_0^\infty \zeta^{-\frac{N}{4}} \mathcal M_\al (\zeta)  \big| D^k \overline{\mathscr B}_0 (x  (t^\al \zeta  )^{-\frac{1}{4}} ) \big|  d\zeta
\end{align}
It follows from {$|\vartheta|^k \le N^{k-1}\sum_{j=1}^N |\vartheta_j|^k $} that
\begin{align} \label{d22}
\Big| D^k  \overline{\mathscr B}_0 \big(x  (t^\al \zeta  )^{-\frac{1}{4}} )  \Big| &\le {N^{k-1}} \sum_{j=1}^N \Bigg| \int_{\mathbb R^N} \parentheses{i\vartheta_j t^{-\frac{\al}{4}}  \zeta^{-\frac{1}{4}}}^k e^{i <\vartheta,x> t^{-\frac{\al}{4}}  \zeta^{-\frac{1}{4}}} e^{-|\vartheta|^4} d\vartheta \Bigg|\nn\\
&\le {N^{k-1}} t^{-\frac{\al k}{4}}  \zeta^{-\frac{k}{4}} \int_{\mathbb R^N} |\vartheta|^k  \big|e^{i <\vartheta,x> t^{-\frac{\al}{4}}  \zeta^{-\frac{1}{4}}} e^{-|\vartheta|^4} \big| d\vartheta.
\end{align}
Then, by using a new variable $x  (t^\al \zeta)^{-\frac{1}{4}} =z$ and some simple caculations, we find that
\begin{equation} \label{d3}
\int_{\mathbb R^N} |\vartheta|^k \big| e^{i <\vartheta,x> t^{-\frac{\al}{4}}  \zeta^{-\frac{1}{4}}} e^{-|\vartheta|^4} \big| d\vartheta= \overline{\mathscr B}_k \big( z\big) .
\end{equation}
Combining \eqref{d1}, \eqref{d22}, \eqref{d3}, 
\begin{align*}
\Big\| D^k \mathbb K_{1,\al} (t,x) \Big\|_{L^p}&= \Bigg( \int_{\mathbb R^N} \Big|D^k \mathbb K_{1,\al} (t,x) \Big|^p dx \Bigg)^{\frac{1}{p}}\nn\\
&\le {N^{k-1}} \Bigg( \int_{\mathbb R^N} \Bigg| t^{-\frac{\al N}{4}-\frac{\al k}{4}}     \int_0^\infty  \mathcal M_\al (\zeta) \zeta^{-\frac{k}{4}-\frac{N}{4}} \overline {\mathscr B}_k \big( z\big)    d\zeta\Bigg|^p dx \Bigg)^{\frac{1}{p}}\\
&\le {N^{k-1}} t^{-\frac{\al N}{4}-\frac{\al k}{4}}   \Bigg( \int_{\mathbb R^N} \Big|    \int_0^\infty \mathcal M_\al (\zeta) \zeta^{-\frac{k}{4}-\frac{N}{4}}  \overline {\mathscr B}_k \big( z\big)     d\zeta \Big|^p   dx \Bigg)^{\frac{1}{p}}.
\end{align*}
Applying  Minkowski’s inequality in integral form and noting that $dx=  (t^\al \zeta )^{\frac{N}{4}} dz$,  
	\begin{align} \label{f1}
		\Big\| D^k  \mathbb K_{1,\al} (t,x) \Big\|_{L^p} &\le{N^{k-1}}   t^{-\frac{\al N}{4}-\frac{\al k}{4}}   \Bigg( \int_{\mathbb R^N} \Big|    \int_0^\infty \mathcal M_\al (\zeta) \zeta^{-\frac{k}{4}-\frac{N}{4}}  \overline {\mathscr B}_k \big( z\big)     d\zeta \Big|^p   dx \Bigg)^{\frac{1}{p}}\nn\\
		&\le{N^{k-1}} t^{-\frac{\al N}{4}-\frac{\al k}{4}}    \int_0^\infty \Bigg( \int_{\mathbb R^N} \Big|     \mathcal M_\al (\zeta) \zeta^{-\frac{k}{4}-\frac{N}{4}}  \overline {\mathscr B}_k \big( z\big)     \Big|^p (t^\al \zeta)^{\frac{N}{4}} dz\Bigg)^{\frac{1}{p}} d\zeta \nn\\
		&\le{N^{k-1}}  t^{-\frac{\al N}{4}-\frac{\al k}{4}+ \frac{\al N}{4p}}   \int_0^\infty \zeta^{-\frac{k}{4}-\frac{N}{4}+\frac{ N}{4p}} \mathcal M_\al (\zeta)       \Bigg( \int_{\mathbb R^N} \Big| \overline{ \mathscr B}_k (z ) \Big| ^p dz \Bigg)^{\frac{1}{p}} d\zeta. 
	\end{align}
Due to the condition $1- \frac{k}{4}-\frac{N}{4}+\frac{N}{4p}>0$ and Lemma \ref{lem21}, we obtain
	$$
	\int_0^\infty \zeta^{-\frac{k}{4}-\frac{N}{4}+\frac{N}{4p}} \mathcal M_\al (\zeta)  d\zeta= \frac{\Gamma \left(1- \frac{k}{4}-\frac{N}{4}+\frac{N}{4p} \right)}{\Gamma \left( \frac{-\al k}{4}-\frac{\al N}{4}+\frac{\al N}{4p}+1 \right)}
	$$
and, together  with  \eqref{f1},  allow us to deduce the following boundness result
\begin{align*}
\Big\| D^k  \mathbb K_{1,\al} (t,x) \Big\|_{L^p} \le \mathscr C_{k,p} (\al, N)   t^{-\frac{\al N}{4}-\frac{\al k}{4}+ \frac{\al N}{4p}} , 
\end{align*}
	where we denote 
	\[
	\mathscr C_{k,p} (\al, N) = \frac{{N^{k-1}}\Gamma \left(1- \frac{k}{4}-\frac{N}{4}+\frac{N}{4p} \right)}{\Gamma \left( \frac{-\al k}{4}-\frac{\al N}{4}+\frac{\al N}{4p}+1 \right)} \Bigg( \int_{\mathbb R^N} \Big|  \overline{\mathscr B}_k (z) \Big|^p dz \Bigg)^{\frac{1}{p}}.
	\] 
	
\noindent\textbf{\textit{Step 2. Verify the second inequality}}

The  representation  of $\mathbb K_{2,\al} (t,x)$ and  \eqref{e1} together with Fubini's theorem imply
\begin{align}\label{a1}
\mathbb K_{2,\al} (t,x)=  \mathcal F^{-1} \Big(t^{\al-1} E_{\al,\al} (-t^\al |\xi|^4)\Big)&= t^{\al-1} \int_{\mathbb R^N}  e^{-i <\xi, x>} \Big(E_{\al,\al} (-t^\al |\xi|^4)\Big) d\xi\nn\\
&=t^{\al-1}  \parentheses{2\pi}^{-N}\int_{\mathbb R^N}  e^{i <\xi, x>} \Big(\int_0^\infty \al\zeta \mathcal M_\al (\zeta) e^{-t^\al \zeta |\xi|^4} d\zeta \Big) d\xi \nn\\
&= t^{\al-1}  \parentheses{2\pi}^{-N}\int_0^\infty \al\zeta  \mathcal M_\al (\zeta)  \Bigg( \int_{\mathbb R^N}  e^{i <\xi, x>}e^{-t^\al \zeta |\xi|^4} d\xi \Bigg)  d\zeta.
\end{align}
By setting  $\xi = \vartheta  (t^\al \zeta )^{-\frac{1}{4}}$, we deduce $d\xi = (t^\al \zeta )^{-\frac{N}{4}} d\vartheta  $ and $|\xi|^4= (t^\al \zeta )^{-1} |\vartheta|^4$.	Using  \eqref{a1}, 
\begin{align*}
\mathbb K_{2,\al} (t,x)=  \al t^{\al-1}  \int_0^\infty  \zeta  (t^\al \zeta)^{-\frac{N}{4}}  \mathcal M_\al (\zeta)  \int_{\mathbb R^N}   e^{i <\vartheta,  x> t^{-\frac{\al}{4}}\eta^{-\frac{1}{4}}}   e^{-|\vartheta|^4} d \vartheta    d\zeta.
\end{align*}
By a similar argument as in Step 1, 
\begin{align*}
\Big\|\mathbb K_{2,\al} (t,x)\Big\|_{L^p} &\le  \al t^{\al-\frac{\al N}{4}-1} \int_0^\infty    \Bigg( \int_{\mathbb R^N} \Bigg| \zeta^{1-\frac{N}{4}}  \mathcal M_\al (\zeta)     \overline{\mathscr B}_0 (x  (t^\al  \zeta)^{-\frac{1}{4}} ) \Bigg| ^p dx \Bigg)^{\frac{1}{p}} d\zeta \nn\\
&=   \al t^{\al-\frac{\al N}{4}-1} \int_0^\infty    \Bigg( \int_{\mathbb R^N} \Big| \zeta^{1-\frac{N}{4}} \mathcal M_\al (\zeta)     \overline{\mathscr B}_0 \left(x  (t^\al \zeta)^{-\frac{1}{4}} \right) \Big| ^p (t^\al \zeta)^{\frac{N}{4}}  dz \Bigg)^{\frac{1}{p}} d\zeta \nn\\
&=  \al t^{\al-\frac{\al N}{4}-1+\frac{\al N}{4p}}      \Bigg( \int_{\mathbb R^N}   \Big|  \overline{\mathscr B}_0 (z) \Big|^p dz \Bigg)^{\frac{1}{p}} \int_0^\infty \zeta^{1+\frac{N}{4p}-\frac{N}{4}}  \mathcal M_\al (\zeta)      d\zeta  \nn\\
&= \frac{\al \Theta_{p,N} \Gamma(\frac{N}{4p}-\frac{N}{4}+2)  }{\Gamma(\frac{\al N}{4p}-\frac{\al N}{4}+1+\al) }  t^{\al-\frac{\al N}{4}-1+\frac{\al N}{4p}}.
\end{align*}
Now, we  estimate the derivative of the quantity $\mathbb K_{2,\al}$. Let us recall the following formula 
\begin{align*}
\mathbb K_{2,\al} (t,x)& =t^{\al-1-\frac{\al N}{4}}  \int_0^\infty \al \zeta^{1-\frac{N}{4}} \mathcal M_\al (\zeta)  \overline{\mathscr B}_0 \big(xt^{-\frac{\al}{4}}  \zeta^{-\frac{1}{4}} \big)d\zeta.
\end{align*}
In view of the boundedness of  $D \mathbb K_{1,\al} (t,x)$, we have 
\begin{align} \label{d111}
\big|D^k \mathbb K_{2,\al} (t,x)\big| \le 	t^{\al-1-\frac{\al N}{4}}   \int_0^\infty \al \zeta^{1-\frac{N}{4}} \mathcal M_\al (\zeta)  \big| D^k \overline{\mathscr B}_0 (x  (t^\al \zeta  )^{-\frac{1}{4}} ) \big|  d\zeta.
\end{align}
It follows from $|\vartheta|^k{\le N^{k-1}}\sum_{j=1}^N |\vartheta_j|^k $ that
\begin{align} \label{d2}
\Big| D^k \overline{\mathscr B}_0 \big(xt^{-\frac{1}{4}}  \eta^{-\frac{1}{4}} \big)  \Big| &\le{N^{k-1}} \sum_{j=1}^N \Bigg| \int_{\mathbb R^N} \parentheses{i\vartheta_j t^{-\frac{\al}{4}}  \zeta^{-\frac{1}{4}}}^k e^{i <\vartheta,x> t^{-\frac{\al}{4}}  \zeta^{-\frac{1}{4}}} e^{-|\vartheta|^4} d\vartheta \Bigg|\nn\\
&\le {N^{k-1}}t^{-\frac{\al k}{4}}  \zeta^{-\frac{k}{4}} \int_{\mathbb R^N} |\vartheta|^k  \big|e^{i <\vartheta,x> t^{-\frac{\al}{4}}  \zeta^{-\frac{1}{4}}} e^{-|\vartheta|^4} \big| d\vartheta.
\end{align}
By using subtitution $x  (t^\al \zeta)^{-\frac{1}{4}} =z$, the  second derivative with respect to $x$ of $\mathbb K_{2,\al} (t,x)$ is estimated by
\begin{align*}
&\Big\| D^k  \mathbb K_{2,\al} (t,x) \Big\|_{L^p}= \Bigg( \int_{\mathbb R^N} \Big| D^k  \mathbb K_{2,\al} (t,x) \Big|^p dx \Bigg)^{\frac{1}{p}}\nn\\
&\le {N^{k-1}} \Bigg( \int_{\mathbb R^N} \Bigg| t^{\al-1-\frac{\al N}{4}-\frac{\al k}{4}}   \int_0^\infty \al \zeta^{1-\frac{N}{4}-\frac{k}{4}} \mathcal M_\al (\zeta)  \int_{\mathbb R^N} |\vartheta|^k \Big| \exp\Big({i <\vartheta,x> t^{-\frac{1}{4}}  \zeta^{-\frac{1}{4}}}\Big) e^{-|\vartheta|^4} \Big| d\vartheta     d\zeta \Bigg|^p dx\Bigg)^{\frac{1}{p}}\nn\\
&={N^{k-1}} \Bigg( \int_{\mathbb R^N} \Bigg| t^{\al-1-\frac{\al N}{4}-\frac{\al k}{4}}   \int_0^\infty  \al \mathcal M_\al (\zeta) \zeta^{1-\frac{N}{4}-\frac{ k}{4}} \overline {\mathscr B}_k \big( z\big)    d\zeta\Bigg|^p dx \Bigg)^{\frac{1}{p}} \nn\\
& \le{N^{k-1}} t^{\al-1-\frac{\al N}{4}-\frac{\al k}{4}}  \Bigg( \int_{\mathbb R^N} \Big|    \int_0^\infty \al \mathcal M_\al (\zeta) \zeta^{1-\frac{N}{4}-\frac{ k}{4}}  \overline {\mathscr B}_k \big( z\big)     d\zeta \Big|^p   dx \Bigg)^{\frac{1}{p}}.
\end{align*}
Applying Minkowski’s inequality in integral form, we find that
\begin{align} \label{e11}
\Big\| D^k \mathbb K_{2,\al} (t,x) \Big\|_{L^p} &\le{N^{k-1}}  t^{\al-1-\frac{\al N}{4}-\frac{\al k}{4}}   \Bigg( \int_{\mathbb R^N} \Big|    \int_0^\infty \al \mathcal M_\al (\zeta) \zeta^{1-\frac{N}{4}-\frac{ k}{4}}  \overline {\mathscr B}_k \big( z\big)     d\zeta \Big|^p   dx \Bigg)^{\frac{1}{p}} \nn\\
&\le{N^{k-1}} \al t^{\al-1-\frac{\al N}{4}-\frac{\al k}{4}}   \int_0^\infty \Bigg( \int_{\mathbb R^N} \Big|     \mathcal M_\al (\zeta) \zeta^{1-\frac{N}{4}-\frac{ k}{4}} \overline {\mathscr B}_k \big( z\big)     \Big|^p (t^\al \zeta)^{\frac{N}{4}} dz\Bigg)^{\frac{1}{p}} d\zeta \nn\\
&\le{N^{k-1}}  \al t^{\al-\frac{\al N}{4}-1+\frac{\al N}{4p}-\frac{\al k}{4}} \int_0^\infty \zeta^{1-\frac{N}{4}+\frac{N}{4p}-\frac{ k}{4}} \mathcal M_\al (\zeta)       \Bigg( \int_{\mathbb R^N} \Big| \overline{ \mathscr B}_k (z ) \Big| ^p dz \Bigg)^{\frac{1}{p}} d\zeta \nn\\
&={N^{k-1}}\al \Bigg( \int_{\mathbb R^N} \Big|  \overline{\mathscr B}_k (z) \Big|^p dz \Bigg)^{\frac{1}{p}}  t^{\al-\frac{\al N}{4}-1+\frac{\al N}{4p}-\frac{\al k}{4}}  \int_0^\infty \zeta^{1-\frac{N}{4}+\frac{N}{4p}-\frac{ k}{4}}  \mathcal M_\al (\zeta)  d\zeta. 
\end{align}
Let us continue by computing the integral term on the right-hand side of \eqref{e1}. Indeed, using Lemma \eqref{lem21} and  noting that $2-\frac{N}{4}+\frac{N}{4p}-\frac{ k}{4}>0$, we immediately derive
	\begin{equation} \label{e12}
		\int_0^\infty \zeta^{1-\frac{N}{4}+\frac{N}{4p}-\frac{ k}{4}}  \mathcal M_\al (\zeta)  d\zeta= \frac{\Gamma \left( 2-\frac{N}{4}+\frac{N}{4p}-\frac{ k}{4} \right)}{\Gamma \left( 1+\al-\frac{\al N}{4}+\frac{\al N}{4p}-\frac{ \al k}{4} \right)}.
	\end{equation}
	Combining \eqref{e11} and \eqref{e12}, we find that there exists $\overline{ \mathscr C_{k,p}} (\al, N)$ such that 
	\begin{align*}
		\Big\| D^k \mathbb K_{2,\al} (t,x) \Big\|_{L^p} \le {N^{k-1}}\Bigg( \int_{\mathbb R^N} \Big|  \overline{\mathscr B}_k (z) \Big|^p dz \Bigg)^{\frac{1}{p}} \frac{\al \Gamma \left( 2-\frac{N}{4}+\frac{N}{4p}-\frac{ k}{4} \right)}{\Gamma \left( 1+\al-\frac{\al N}{4}+\frac{\al N}{4p}-\frac{ \al k}{4} \right)}   t^{\al-\frac{\al N}{4}-1+\frac{\al N}{4p}-\frac{\al k}{4}} .
	\end{align*}
We end the proof here.
\end{proof}
{\begin{remark}\label{Functional spaces remarks}
Let us state some comments on the assumptions of the above lemma as follows.
\begin{enumerate}[(i)]
\item When $ p=1 $ the assumtion $ k<4-N\left(1-\frac{1}{p}\right) $ implies that we can take $ k $ from the set $ \{0,1,2,3\} $. In addition, when $ p=1 $ and $ k=0 $, from the facts that $ \Theta_{1,N}=1,~\frac{\alpha}{\Gamma(1+\alpha)}<\Gamma(2)=1 $, we can bound $ \mathscr C_{k,p} $ and $ \overline{ \mathscr C_{k,p}} $ by $ 1 $. 	
\item When $ k=0 $, the assumption becomes $ \frac{N}{4}\left(1-\frac{1}{p}\right)<1 $. This assumption is always satisfied whenever $ N\le4 $. On the other hand, when $ N\ge5 $, we need to consider the condition that $ p<\frac{N}{N-4} $ further, if we want to apply this lemma. 
\item When $ k\ge1 $, we have a certain restriction on the hypothesis for $ p $. For example, when $ N=4 $ the hypothesis $ 1\le p<\frac{N}{N-k} $ implies that $ p\in\{1,2,3\}. $ In short, when using Lemma \ref{lemqt},  the larger $ k $ and dimension $ N $, the more restricted on the amount of $ p. $ 
\end{enumerate}
\end{remark}}

\begin{remark}
From \cite[Proposition 2.1]{Galaktionov}, we can bound $ \Theta_{p,N} $ by a constant $\Theta_{N}$ independent of $ p $. This fact will be needed when we set up some linear estimates.
\end{remark}
\subsection{Space setting}
\begin{definition}\label{Original definition}
	{Assume that a function $ \Xi:\R^+\cup\{0\}\rightarrow\R^+\cup\{0\} $ is increasing convex, right continuous at $ 0 $ and} 
	\begin{align*}
		\lim\limits_{z\to\infty}\Xi(z)=\infty.
	\end{align*}
Then, we define the Orlicz space $ L^{\Xi}(\R^N) $ in the following fashion
\begin{align*}
L^{\Xi}(\R^N)=\left\{\varphi\in L_{loc}^1(\R^N);\int_{\R^N}\Xi\parentheses{\frac{|\varphi(x)|}{\kappa}}dx<\infty,\text{ for some }\kappa>0\right\}.
\end{align*}
\end{definition}
\begin{remark}
	The Orlicz space $ L^{\Xi}(\R^N) $ mentioned above is a Banach space, endowed with the Luxemburg norm
	\begin{align*}
		\norm{\varphi}_{\Xi}=\inf\left\{\kappa>0;\int_{\R^N}\Xi\parentheses{\frac{|\varphi(x)|}{\kappa}}dx\le1\right\}.
	\end{align*}
\end{remark}

\begin{remark}
Let $ 1<p <\infty $, by choosing $ \Xi(z)=z^p $, we can identify the space $ L^{\Xi}(\R^N) $ with the usual Lebesgue space $ \Lp. $ For the sake of brevity, we set 
\begin{align*}
\norm{\cdot}_{L^p+L^q}:=\norm{\cdot}_{\Lp\cap\Lq},\qquad p,q\in[1;\infty].
\end{align*}
\end{remark}

\begin{definition}\label{Main_Orlicz_Space}
Let $ 1\le p <\infty $, in the rest of this work, we use the symbol $ \Lep $ to indicate the Orlicz spaces with $ \Xi(z)=e^{|z|^p}-1$. We also denote 
\begin{align*}
\norm{\cdot}_{L^q+\Xi}:=\norm{\cdot}_{\Lq\cap\Lep},\qquad q\in[1;\infty].
\end{align*}
\begin{definition}
Let $ 1\le p<\infty $. We define the following subspace of $ \Lep $
\begin{align*}
	L^{\Xi}_0(\R^N)=\left\{\varphi\in L_{loc}^1(\R^N);\int_{\R^N}\Xi\parentheses{\frac{|\varphi(x)|}{\kappa}}dx<\infty,\text{ for every }\kappa>0\right\}.
\end{align*}
\begin{remark}
It can be shown from \cite{Ioku2} that $ \0Lep=\overline{C_0^{\infty}(\R^N)}^{\Lep}. $
\end{remark}
\end{definition}
\end{definition}

From the previous definitions, we can note that the Orlicz space is a generalization of the usual Lebesgue space. Let us introduce some of the useful embeddings between Orlicz spaces and Lebesgue spaces that we will need in our main results section.   
\begin{lemma}\label{Embedding exp to Lq}
	For every constants $ p,q $ satisfying $ 1\le p\le q<\infty $, the embedding $ \Lep\hookrightarrow\Lq $ holds. In addition, 
	\begin{align}
		\norm{\varphi}_{L^q}\le\sqrbrackets{\Gamma\parentheses{\frac{q}{p}+1}}^{\frac{1}{q}}\norm{\varphi}_{\Xi}.
	\end{align}
\end{lemma}
\begin{lemma}\label{Embedding_Lq_Linfty}
Given $ 1\le q\le p $, we have $ \Lq\cap L^{\infty}(\R^N)\hookrightarrow \0Lep\subsetneq\Lep $. In particular, for any $ \varphi\in\Lq\cap L^{\infty}(\R^N) $ the following bound holds
\begin{align*}
\norm{\varphi}_{\Xi}\le(\log2)^{\frac{-1}{p}}\sqrbrackets{\norm{\varphi}_{L^q}+\norm{\varphi}_{L^{\infty}}}.
\end{align*}
\end{lemma}
\begin{lemma}\label{Kernel_apply_to_Orlicz}
Let $ p\ge1 $ and $ \al\in(0,1) $. {Then, we can find constants $ \mathscr{C}_{1,h},\mathscr{C}_{2,h},\mathscr{C}_{\Xi} $ such that the following results hold.}
\begin{enumerate}[(i)]
\item Suppose that $ h\in[1,p] $ satisfies $h>N/4 $, for any $ \varphi\in\Lh $, we have
\begin{align*}
&\norm{\mathbb Z_{1,\al}(t)\varphi}_{\Xi}\le {\mathscr{C}_{1,h}}t^{\frac{-\alpha N}{4h}}\sqrbrackets{\log\parentheses{1+t^{\frac{-\alpha N}{4}}}}^{\frac{-1}{p}}\norm{\varphi }_{L^h},\\
&\norm{\mathbb Z_{2,\al}(t)\varphi}_{\Xi}\le {\mathscr{C}_{2,h}}t^{\frac{-\alpha N}{4q}}\sqrbrackets{\log\parentheses{1+t^{\frac{-\alpha N}{4}}}}^{\frac{-1}{p}}\norm{\varphi }_{L^h}.
\end{align*}
\item For any $ \varphi\in\Lep, $ we have
\begin{align*}
\norm{\mathbb Z_{1,\al}(t)\varphi}_{\Xi}\le  \norm{\varphi }_{\Xi},\qquad\norm{\mathbb Z_{2,\al}(t)\varphi}_{\Xi}\le t^{\alpha-1}\norm{\varphi }_{\Xi}.
\end{align*}
\end{enumerate}
\end{lemma}
\begin{proof}
Firstly, by using Young's convolution inequality, there exists a constant $ q\in[1,\infty] $ such that
\begin{align}
\norm{\mathbb{Z}_{i,\al}(t)\varphi}_{L^p}\le\norm{\mathbb{K}_{i,\al}(t)}_{L^q}\norm{\varphi}_{L^h}.
\end{align}
Then, thanks to Lemma \ref{lemqt}, we have
\begin{align*}
\norm{\mathbb{Z}_{1,\al}(t)\varphi}_{L^p}\le \frac{\Theta_{N}\Gamma\parentheses{1-\frac{ N}{4}\parentheses{\frac{1}{h}-\frac{1}{p}}}}{\Gamma\parentheses{1-\frac{ \alpha  N}{4}\parentheses{\frac{1}{h}-\frac{1}{p}}}}t^{\frac{- \alpha  N}{4}\parentheses{\frac{1}{h}-\frac{1}{p}}}\norm{\varphi}_{L^h},\\
\norm{t^{1-\alpha}\mathbb{Z}_{2,\al}(t)\varphi}_{L^p}=t^{1-\alpha}\norm{\mathbb{Z}_{2,\al}(t)\varphi}_{L^p}\le \frac{\alpha \Theta_{N}\Gamma\parentheses{2-\frac{ N}{4}\parentheses{\frac{1}{h}-\frac{1}{p}}}}{\Gamma\parentheses{2-\frac{ \alpha  N}{4}\parentheses{\frac{1}{h}-\frac{1}{p}}}}t^{\frac{-\alpha  N}{4}\parentheses{\frac{1}{h}-\frac{1}{p}}}\norm{\varphi}_{L^h}.
\end{align*}

We can show that the constants on the right-hand side of the above estimates can be bounded by two constants $ \mathscr{C}_{1,h},\mathscr{C}_{2,h} $ that are independent of $ p $, respectively. In fact, by properties of the Gamma function when $ 0<\frac{N}{4}\parentheses{\frac{1}{h}-\frac{1}{p}}<1<\frac{29}{20} $, we obtain
\begin{align*}
\norm{\mathbb{Z}_{1,\al}(t)\varphi}_{L^p}\le \Theta_{N}\Gamma\parentheses{1-\frac{ N}{4h}}t^{\frac{- \alpha  N}{4}\parentheses{\frac{1}{h}-\frac{1}{p}}}\norm{\varphi}_{L^h}.
\end{align*}
On the other hand, the Gautschi inequality implies
\begin{align*}
\frac{\Gamma\parentheses{2-\frac{N}{4}\parentheses{\frac{1}{h}-\frac{1}{p}}}}{\Gamma\parentheses{2-\frac{\alpha N}{4}\parentheses{\frac{1}{h}-\frac{1}{p}}}}\le\sqrbrackets{1-\frac{\alpha N}{4}\parentheses{\frac{1}{h}-\frac{1}{p}}}^{\frac{(\alpha-1)N}{4}\parentheses{\frac{1}{h}-\frac{1}{p}}}
\le\parentheses{1-\frac{\alpha N}{4h}}^{\frac{(\alpha-1)N}{4h}}.
\end{align*}
It follows that
\begin{align*}
\norm{t^{1-\alpha}\mathbb{Z}_{2,\al}(t)\varphi}_{L^p}\le \al\Theta_{N}\parentheses{1-\frac{\alpha N}{4h}}^{\frac{(\alpha-1)N}{4h}}t^{\frac{-\alpha  N}{4}\parentheses{\frac{1}{h}-\frac{1}{p}}}\norm{\varphi}_{L^h}.
\end{align*}
We are now ready to verify our main statements. Because the techniques are the same, we will present only the proof for the second one, $ \mathbb{Z}_{2,\al}(t)\varphi $. We note that for $ j\ge1 $
\begin{align*}
\norm{t^{1-\alpha}\mathbb{Z}_{2,\al}(t)\varphi}_{L^{pj}}^{pj}
\le \mathscr{C}_{2,h}^{pj}t^{\frac{-\alpha Npj}{4}\parentheses{\frac{1}{h}-\frac{1}{pj}}}\norm{\varphi }_{L^h}^{pj}\le t^{\frac{\alpha N}{4}}\sqrbrackets{\mathscr{C}_{2,h}t^{\frac{-\alpha N}{4h}}\norm{\varphi }_{L^h}}^{pj}.
\end{align*}
Then, the Taylor expansion of the exponential leads us to 
\begin{align}\label{SP-Main Inequality 2}
\int_{\R^N}\sqrbrackets{\exp\parentheses{\frac{\left|t^{1-\alpha}\mathbb{Z}_{2,\al}(t)\varphi(x)\right|^p}{\kappa^p}}-1}dx &=\sumuse\frac{\norm{t^{1-\alpha}\mathbb{Z}_{2,\al}(t)\varphi}_{L^{pj}}^{pj}}{j!\kappa^{pj}}\nn\\
&\le t^{\frac{\alpha N}{4}}\sumuse\frac{\sqrbrackets{\mathscr{C}_{2,h}t^{\frac{-\alpha N}{4h}}\norm{\varphi }_{L^h}}^{pj}}{j!\kappa^{pj}} \\
&\nonumber=t^{\frac{\alpha N}{4}}\sqrbrackets{\exp\parentheses{\frac{\mathscr{C}_{2,h}t^{\frac{-\alpha N}{4h}}\norm{\varphi }_{L^h}}{\kappa}}^p-1}. 
\end{align}
Next, assume that the right hand side of the above estimate is less than or equal $ 1 $. Then, we can easily find that
\begin{align*}
\mathscr{C}_{2,h}t^{\frac{-\alpha N}{4h}}\sqrbrackets{\log\parentheses{1+t^{\frac{-\alpha N}{4}}}}^{\frac{-1}{p}}\norm{\varphi }_{L^h}\le\kappa.
\end{align*}
In view of \eqref{SP-Main Inequality 2}, if we set
\begin{align*}
&A:=\left\{\kappa>0;\int_{\R^N}\sqrbrackets{\exp\parentheses{\frac{\left|\mathbb{Z}_{2,\al}(t)\varphi(x)\right|^p}{\kappa^p}}-1}dx\le1\right\},\\
&B:=\left\{\kappa>0;\mathscr{C}_{2,h}t^{\frac{-\alpha N}{4h}}\sqrbrackets{\log\parentheses{1+t^{\frac{-\alpha N}{4}}}}^{\frac{-1}{p}}\norm{\varphi }_{L^h}\le\kappa\right\},
\end{align*} 
the cover result $ B\subset A $ holds. This implies that 
\begin{align*}
\inf A\le\inf B=\mathscr{C}_{2,h}t^{\frac{-\alpha N}{4h}}\sqrbrackets{\log\parentheses{1+t^{\frac{-\alpha N}{4}}}}^{\frac{-1}{p}}\norm{\varphi }_{L^h}.
\end{align*}
We obtain the first results of this lemma. 	To prove the remaining result, we only need   to modify slightly inequality \eqref{SP-Main Inequality 2} in the following way
\begin{align*}
\int_{\R^N}\sqrbrackets{\exp\parentheses{\frac{\left|t^{1-\alpha}\mathbb{Z}_{2,\al}(t)\varphi(x)\right|^p}{\kappa^p}}-1}dx&=\sumuse\frac{\norm{t^{1-\alpha}\mathbb{Z}_{2,\al}(t)\varphi}_{L^{pj}(\R^N)}^{pj}}{j!\kappa^{pj}}\\
&\le \sumuse\frac{\norm{\varphi }_{L^{pj}(\R^N)}^{pj}}{j!\kappa^{pj}}=\exp\parentheses{\frac{\norm{\varphi }_{L^p}}{\kappa}}^p-1. 
\end{align*}
Then, our statements follow.
\end{proof}

\begin{proposition}\label{ContinuityProposition}
Assume that $ \varphi\in\0Lep $. Then, we have
\begin{align*}
\mathbb{Z}_{1,\al}(t)\varphi\in C\parentheses{[0,T];\0Lep}.
\end{align*}	
\end{proposition}
\begin{proof}
Since $ \varphi\in\0Lep $, there exists a sequence $ \{\varphi_n\}_{n\in\N}\subset C_{0}^{\infty}(\R^N) $ such that $\varphi_n$ converges to $ \varphi $ with respect to $\Lep$ norm. This implies that, for any $ t>0 $, $ \mathbb{Z}_{1,\al}(t)\varphi_n$ will converge to $ \mathbb{Z}_{1,\al}(t)\varphi. $ Indeed, by applying Lemma \ref{Kernel_apply_to_Orlicz}, we have
\begin{align*}
\norm{\mathbb{Z}_{1,\al}(t)\varphi_n-\mathbb{Z}_{1,\al}(t)\varphi}_{\Xi}\le\norm{\varphi_n-\varphi}_{\Xi}\xrightarrow{n\rightarrow\infty}0.
\end{align*}
By taking two number $ t_1,t_2>0 $,  the triangle inequality implies
\begin{align*}
\begin{array}{lcr}
&\dis\norm{\mathbb{Z}_{1,\al}(t_2)\varphi-\mathbb{Z}_{1,\al}(t_1)\varphi}_{\Xi}&\\[0.3cm]
&\dis\le\norm{\mathbb{Z}_{1,\al}(t_2)\varphi_n-\mathbb{Z}_{1,\al}(t_2)\varphi}_{\Xi}			+\norm{\mathbb{Z}_{1,\al}(t_1)\varphi_n-\mathbb{Z}_{1,\al}(t_1)\varphi}_{\Xi}&\\[0.3cm]
&\dis+\norm{\mathbb{Z}_{1,\al}(t_2)\varphi_n-\mathbb{Z}_{1,\al}(t_1)\varphi_n}_{\Xi}&.
\end{array}
\end{align*}
Combining Lemma \ref{Kernel_apply_to_Orlicz}, the definition of $ \0Lep $ and the application of Theorem \ref{Continuous_uni_operator_topology} {for $ \Lp $ and $ C_0(\mathbb{R}^N) $}
{\begin{align*}
\begin{cases}
\dis\lim\limits_{t_2\to t_1}\norm{\mathbb{Z}_{1,\al}(t_2)\varphi_n-\mathbb{Z}_{1,\al}(t_1)\varphi_n}_{L^p}=0,\quad&\varphi_n\in C_{0}^{\infty}(\R^N),\\[0.3cm]
\dis\lim\limits_{t_2\to t_1}\norm{\mathbb{Z}_{1,\al}(t_2)\varphi_n-\mathbb{Z}_{1,\al}(t_1)\varphi_n}_{L^\infty}=0,&\varphi_n\in C_{0}^{\infty}(\R^N),
\end{cases}
\end{align*}}
and  
\begin{align*}
\begin{cases}
\dis\lim\limits_{n\to\infty}\norm{\mathbb{Z}_{1,\al}(t_1)\varphi_n-\mathbb{Z}_{1,\al}(t_1)\varphi}_{\Xi}=0,\\[0.3cm]
\dis\lim\limits_{n\to\infty}\norm{\mathbb{Z}_{1,\al}(t_2)\varphi_n-\mathbb{Z}_{1,\al}(t_2)\varphi}_{\Xi}=0.
\end{cases}
\end{align*}
Consequently, by an appropriate choice of $ n $, the desired conclusion of this proposition can be drawn easily.
\end{proof}

\section{Time fractional biharmonic equation with exponential nonlinearity}
In this section, we investigate the IVP for the time-fractional biharmonic equation with exponential nonlinearity
\begin{align}\tag{P1}\label{Main Problem}
\begin{cases}
\dis \partial^\al_{0|t}\p(t,x)+\Delta^2\p(t,x)=G(u(t,x)),\qquad &\mbox{in}\quad\R^+\times\R^N,\\
\dis\p(0,x)=\p_0(x),&\mbox{in} \quad \R^N,
\end{cases}
\end{align} 
with the following assumptions of the initial function: 

\noindent\textit{\textbf{Assumption 1}}: The initial function $ \p_0 $ belongs to  $ \Lp\cap {C_0(\R^N)} $. 

\noindent\textit{\textbf{Assumption 2}}: The  initial function $ \p_0 $  belongs to $ \Lep $ or $ \0Lep$. 

\subsection{Unique existence of mild solution under the first assumption for the initial function}
In this part, we investigate Problem \eqref{Main Problem} with the assumption that  $ u_0 $ belongs to the space $ \Lp\cap{C_0(\R^N)} $.  
\begin{theorem}\label{LocalExistenceTheoremCase1}
	Assume that $ \p_0\in\Lep\cap{C_0(\R^N)} $.  Then, there exists a unique solution of 
	Problem \eqref{Main Problem} that belongs to $ C\parentheses{(0,T];\Lep\cap{C_0(\R^N)}}. $ 
\end{theorem}
\begin{proof}
The proof is begun by fixing a constant $\Im>0 $ and choosing a small time $ T<\sqrt[\al]{\al\Im_1^{-1}} $, where $ \Im_1 $ is defined in \eqref{definition of Im_1}. Next, we consider the following space
\begin{align*}
\mathbf{A}:=\left\{u\in C\parentheses{(0,T];\Lep\cap{C_0(\R^N)}};\sup_{t\in(0,T]}\norm{u(t)-\mathbb{Z}_{1,\al}(t)\p_{0}}_{L^{\infty}+\Xi}\le\Im \right\},
\end{align*}
and the operator $ \mathscr F:\mathbf{A}\rightarrow\mathbf{A} $ given by
\begin{align}\label{F formula}
\mathscr{F}u(t)&=\mathbb{Z}_{1,\al}(t)u_{0}+\int_{0}^{t}\mathbb{Z}_{2,\al}(t-\tau)G(u(\tau))d\tau.
\end{align}
By Young's convolution inequality and Lemma \ref{lemqt}, we have
\begin{align*}
\Bigg\|\int_{\mathbb R^N}   \mathbb K_{1,\al} (t,\cdot-y) u_0(y)dy\Bigg\|_{L^\infty}&=\Big\|  \mathbb K_{1,\al} (t) * u_0 \Big\|_{L^\infty}\le \|  \mathbb K_{1,\al} (t) \|_{L^1(\mathbb R^N)}  \|   u_0 \|_{L^\infty} \le  \|u_0 \|_{L^\infty}. 
\end{align*}
Then, for every $ u\in\mathbf{A} $, the following estimate holds
\begin{align}
\norm{u(t)}_{L^{\infty}+\Xi}&\le\norm{\mathbb{Z}_{1,\al}(t)\p_{0}}_{L^{\infty}+\Xi}+\norm{u(t)-\mathbb{Z}_{1,\al}(t)\p_{0}}_{L^{\infty}+\Xi}\nn\\
&\le2(\log2)^{\frac{-1}{p}}\norm{\p_{0}}_{L^{\infty}+\Xi}+\Im=:\Im_0.
\end{align} 
Our main goal is to prove that the integral equation \eqref{F formula} has a unique solution by the fixed point argument. To this end, we present the following two steps.
	
\noindent\textbf{Step 1}. From \eqref{Property of J}, for any $ u\in\mathbb{A} $ and $ t>0 $, we deduce
\begin{align}\label{Boundary of J case 1}
\begin{cases}
\dis\norm{G(u(t))}_{L^{\infty}}\le L\norm{u(t)}^{m}_{L^{\infty}}e^{\kappa\norm{u(t)}^p_{L^{\infty}}}\le L\Im_0^me^{\kappa\Im_0^p},\\
\dis\norm{G(u(t))}_{L^p}\le L\norm{u(t)}^{m}_{L^{mp}}e^{\kappa\norm{u(t)}^p_{L^\infty}}\le L\Im_0^me^{\kappa\Im_0^p}\parentheses{\Gamma(m+1)}^{\frac{1}{p}},
\end{cases}
\end{align}
where we have used Lemma \ref{Embedding exp to Lq} to achieve the second inequality. It follows that
\begin{align}\label{Boundary of J}
\nonumber\norm{G(u(t))}_{L^\infty+\Xi}&\le(\log2)^{\frac{-1}{p}}\norm{{G}(u(t))}_{L^p+L^\infty}+\norm{{G}(u(t))}_{L^\infty}\le L\parentheses{2+\parentheses{\Gamma(m+1)}^{\frac{1}{p}}}\Im_0^m(\log2)^{\frac{-1}{p}}e^{\kappa\Im_0^p}.
\end{align} 
In addition, by using Lemma \ref{Embedding_Lq_Linfty} and applying Theorem \ref{Continuous_uni_operator_topology} with respect to the spaces $ \Lp,C_0(\mathbb{R}^N) $, for any $ \tau,t\in(0,T], $ we deduce   
\begin{align}
t^{1-\al}\mathbb{Z}_{2,\al}(t)G(u(\tau))\in C\parentheses{(0,T];\Lep\cap{C_0(\R^N)}}.
\end{align}
Next, by taking a positive number $ h $ and $ r\in\{t,t+h\} $, for any $ t\ge\xi $, we obtain
\begin{align*}
(t+h-\tau)^{\alpha-1}\norm{(r-\tau)^{1-\al}\mathbb{Z}_{2,\al}(r-\tau)G(u(\tau))}_{L^\infty+\Xi}\le(t-\tau)^{\alpha-1}\norm{ G(u(\tau))}_{L^\infty+\Xi}\le C(t-\tau)^{\alpha-1}.
\end{align*}
Then, applying the Lebesgue dominated convergence theorem, 
\begin{align}\label{Case1Limit1}
\lim\limits_{h\to 0}\norm{\int_{0}^{t}(t+h-\xi)^{\alpha-1}\sqrbrackets{\frac{\mathbb{Z}_{2,\al}(t+h-\tau)}{(t+h-\tau)^{\al-1}}-\frac{\mathbb{Z}_{2,\al}(t-\tau)}{(t-\tau)^{\al-1}}}G(u(\tau))d\tau}_{L^\infty+\Xi}=0. 
\end{align}
On the other hand, using the fact that $ \lim\limits_{h\to 0}\left|(t+h)^{\alpha}-h^{\alpha}-t^{\alpha}\right|=0 $ and  \eqref{Boundary of J}, we further find that 
\begin{align}\label{Case1Limit2}
&\lim\limits_{h\to 0}\norm{\int_{0}^{t}\sqrbrackets{(t+h-\xi)^{\alpha-1}-(t-\xi)^{\alpha-1}}\frac{\mathbb{Z}_{2,\al}(t+h-\tau)}{(t+h-\tau)^{\al-1}}G(u(\tau))d\tau}_{L^\infty+\Xi}=0.
\end{align}
For the purpose of proving {that} the integral term on the RHS of \eqref{F formula} is continuous on $ (0,T] $, we also need the upcoming fact 
\begin{align*}
&\int_{t}^{t+h}\norm{\mathbb{Z}_{2,\al}(t+h-\tau)G(u(\tau))d\tau}_{L^\infty+\Xi}d\tau\le Ch^{\al}\xrightarrow{\quad h\rightarrow0\quad}0.
\end{align*}
Taking into consideration the above limit results and applying Theorem \ref{Continuous_uni_operator_topology} to $ \p_{0}\in\Lp\cap{C_0(\R^N)}$, we can claim that $ \mathscr{F}u(t) $ is a continuous mapping on $ (0,T].$
 
\noindent\textbf{Step 2.} By using the H\"older inequality, for a constant $ \aleph\in[1,\infty) $ and $ u_1,u_2\in\mathbf{A} $,  we have
\begin{align*}
\norm{G(u_1(\tau))-G(u_2(\tau))}_{L^p}&\le L\sqrbrackets{\sum_{j=1,2}\norm{|u_j(\tau)|^{m-1}e^{\kappa|u_j(\tau)|^p}}_{L^{\frac{\aleph p}{\aleph-1}}}}\norm{u_1(\tau)-u_2(\tau)}_{L^{\aleph p}}.
\end{align*}
Since the embedding $ \Lep\hookrightarrow\Lq $ holds for any $ q\in[p,\infty) $, and $ u_j(\tau)\in{C_0(\R^N)} ( j=1,2) $, the above inequality becomes
\begin{align*}
\norm{G(u_1(\tau))-G(u_2(\tau))}_{L^p}&\le2L\parentheses{\Gamma(\aleph+1)}^{\frac{1}{\aleph p}}\parentheses{\Gamma\parentheses{\frac{\aleph(m-1)}{\aleph-1}+1}}^{\frac{\aleph-1}{\aleph p}}\Im_0^{m-1}e^{\kappa\Im_0^p}\norm{u_1(\tau)-u_2(\tau)}_{\Xi}.
\end{align*}
On the other hand, it is a simple matter to check that
\begin{align*}
\norm{G(u_1(\tau))-G(u_2(\tau))}_{L^\infty}&\le L\sqrbrackets{\sum_{j=1,2}\|u_j(\tau)\|_{L^\infty}^{m-1}e^{\kappa\|u_j(\tau)\|_{L^\infty}^p}}\norm{u_1(\tau)-u_2(\tau)}_{L^\infty}\nonumber\\
&\le2L \Im_0 ^{m-1}e^{\kappa \Im_0 ^p}\norm{u_1(\tau)-u_2(\tau)}_{L^\infty}.
\end{align*}
The two above results help us to deduce
\begin{align}\label{definition of Im_1}
&\norm{G(u_1(\tau))-G(u_2(\tau))}_{L^\infty+\Xi}\nonumber\\
&\le2L\sqrbrackets{2+\parentheses{\Gamma(\aleph+1)}^{\frac{1}{\aleph p}}\parentheses{\Gamma\parentheses{\frac{\aleph(m-1)}{\aleph-1}+1}}^{\frac{\aleph-1}{\aleph p}}}(\log2)^{\frac{-1}{p}}\nonumber \Im_0 ^{m-1}e^{\kappa \Im_0 ^p}\norm{u_1(\tau)-u_2(\tau)}_{L^\infty+\Xi}\\
&=:\Im_1\norm{u_1(\tau)-u_2(\tau)}_{L^\infty+\Xi}.
\end{align}
Also, Lemma \ref{lemqt} shows us that
\begin{align*}
&\Bigg\|\int_{\mathbb R^N}   \mathbb K_{2,\al} (t-\tau,x-y) \sqrbrackets{G(u_1(\tau,y))-G(u_2(\tau,y))}dy\Bigg\|_{L^\infty} \\ &\le \|\mathbb K_{2,\al} (t-\tau,x) \|_{L^1(\mathbb R^N)}\norm{G(u_1(\tau))-G(u_2(\tau))}_{L^{\infty}}\nonumber\\ &\le(t-\tau)^{\al-1}\norm{G(u_1(\tau))-G(u_2(\tau))}_{L^{\infty}}. 
\end{align*}
From this result and Lemma \ref{Kernel_apply_to_Orlicz}, we find that  
\begin{align*}
\norm{\mathscr{F}u_1(t)-\mathscr{F}u_2(t)}_{L^\infty+\Xi}&\le\int_{0}^{t}\norm{\mathbb{Z}_{2,\al}(t-\tau)\sqrbrackets{G(u_1(\tau))-G(u_2(\tau))}}_{L^\infty+\Xi}d\tau\nonumber\\
&\le\int_{0}^{t}(t-\tau)^{\alpha-1}\norm{G(u_1(\tau))-G(u_2(\tau))}_{L^\infty+\Xi}d\tau \\
&\le \frac{T^{\al}\Im_1}{\al}\sup_{t \in(0, T]}\norm{u_1(t)-u_2(t)}_{L^\infty+\Xi}.
\end{align*}
{By choosing $ u_2\equiv0 $ and $ u_1\in\mathbf{A} $ in Step 2, we derive
\begin{align*}
\sup_{t \in(0, T]}\norm{\mathscr{F}u_1(t)-\mathbb{Z}_{1,\al}(t)\p_{0}}_{L^{\infty}+\Xi}\le\frac{T^{\al}\Im_1}{\al}\sup_{t \in(0, T]}\norm{u_1(t)}_{L^\infty+\Xi}\le\Im.
\end{align*}}
{This result along with Step 1 show that $ \mathscr{F} $ is invariant on $ \mathbf{A} $}. Furthermore, it is obvious that $\mathbf{A}$ is a complete metric space with the metric
\begin{align*}
d(u,v):=\sup_{t \in(0, T]}\norm{u(t)-v(t)}_{L^\infty+\Xi}.
\end{align*}

Therefore, the Banach principle argument can be applied to conclude that $ \mathscr{F} $ has a unique fixed point, which means that there exists a unique solution of the Problem \eqref{Main Problem} belonging to $ \mathbf{A} $.
\end{proof}
\begin{lemma}\label{ExtensionLemma}
Let $ \p:(0,T]\rightarrow\Lep\cap{C_0(\R^N)} $ be the unique mild solution of Problem \eqref{Main Problem}. Then there exists a unique continuous extension $ \p^* $ of $u$ on $ (0,T+h], $ for some $ h>0. $
\end{lemma}
\begin{proof}
The main idea of the proof is to show that there exists a unique solution of \eqref{Main Problem} which belongs to the space
$$\mathbf{B}=\left\{ {w\in C\parentheses{(0,T+h];\Lep\cap{C_0(\R^N)}}\Biggl|\begin{array}{*{20}{c}}
{w(t)=u(t),\forall t\in(0,T]}\\
{\sup_{t \in[ T,T+h]}\norm{w(t)-\p(T)}_{L^\infty+\Xi}\le\Im}
\end{array}} \right\},$$
where $ \Im,h>0 $ will be chosen later. To this end, let us consider the function $ \mathscr{J}:\mathbf{B}\rightarrow\mathbf{B}, $  satisfying
\begin{align*}\label{J formula}
\mathscr{J}w(t)&=\mathbb{Z}_{1,\al}(t)u_{0}+\int_{0}^{t}\mathbb{Z}_{2,\al}(t-\tau)G(u(\tau))d\tau.
\end{align*}
It is easily seen that for any $ t\in(0,T] $, $ \mathscr{J}(w(t))=\mathscr{J}(\p(t))=\p(t) $. The continuity result of $ \mathscr{J} $  on $ (0,T+h] $ can be drawn by the same arguments as in Theorem \ref{LocalExistenceTheoremCase1}. Our remaining part of proving the well definition of $ \mathscr{J} $ is to find that if $ w\in\mathbf{B} $, $ \sup_{t \in[ T,T+h]}\norm{\mathscr{J}(w(t))-\p(T)}_{L^\infty+\Xi}\le\Im$. Indeed, for any $ t\in[T,T+h],$ we have 
\begin{align*}
\norm{\mathscr{J}(w(t))-\p(T)}_{L^\infty+\Xi}&\le\underbrace{\norm{\mathbb{Z}_{1,\al}(t)\p_{0}-\mathbb{Z}_{1,\al}(T)\p_{0}}_{L^\infty+\Xi}}_{(I)}+\underbrace{\int_{T}^{t}\norm{\mathbb{Z}_{2,\al}(t-\tau)G(u(\tau))}_{L^\infty+\Xi}d\tau}_{(II)}\nonumber\\
&+\underbrace{\int_{0}^{T}\Bnorm{\sqrbrackets{\mathbb{Z}_{2,\al}(t-\tau)-\mathbb{Z}_{2,\al}(T-\tau)}G(u(\tau))}_{L^\infty+\Xi}d\tau}_{(III)}.
\end{align*}

\noindent$\bullet$ Choosing a sufficiently small $ h_1 $ and using Theorem \ref{Continuous_uni_operator_topology}, Lemma \ref{Embedding_Lq_Linfty}, for every $ t\in[T,T+h_1], $ we can find that $ (I)\le\frac{\Im}{3}. $ 
	
\noindent$\bullet$ Choosing a sufficiently small $ h_2 $, for every $ t\in[T,T+h_2], $ we obtain
\begin{align*}
(II)&\le \frac{L\parentheses{2+\parentheses{\Gamma(m+1)}^{\frac{1}{p}}}\parentheses{\norm{\p(T)}_{L^\infty+\Xi}+\Im}^m}{(\log2)^{\frac{1}{p}}e^{-\kappa\parentheses{\norm{\p(T)}_{L^\infty+\Xi}+\Im}^p}} \int_{T}^{T+h_2}(t-\tau)^{\alpha-1}d\tau \\
&\le C\sqrbrackets{(t-T)^{\alpha}-(t-T-h_2)^{\alpha}}\le\frac{\Im}{3}.
\end{align*}
	
\noindent$\bullet$ Choosing a sufficiently small $ h_3 $ and repeating the arguments for \eqref{Case1Limit1},\eqref{Case1Limit2}, we can also find that
\begin{align*}
(III)\le\frac{\Im}{3},\quad\text{for every~}t\in[T,T+h_3].
\end{align*}	

Additionally, $ \mathscr{J} $ is a contraction mapping on $ \mathbf{B} $ if $ h=h_4 $ is small enough. In fact, for $ u_1,u_2\in\mathbf{B},$ we have
\begin{align*}
&\norm{\mathscr{J}u_1(t)-\mathscr{J}u_2(t)}_{L^\infty+\Xi}\le\int_{T}^{t}(t-\tau)^{\alpha-1}\norm{G(u_1(\tau))-G(u_2(\tau))}_{L^\infty+\Xi}\nonumber d\tau\\
&\nonumber\le\frac{2L \sqrbrackets{2+\parentheses{\Gamma(\aleph+1)}^{\frac{1}{\aleph p}}\parentheses{\Gamma\parentheses{\frac{\aleph(m-1)}{\aleph-1}+1}}^{\frac{\aleph-1}{\aleph p}}}}{\al h^{-\al}_4(\log2)^{\frac{1}{p}}\parentheses{\norm{\p(T)}+\Im}^{1-m}e^{-\kappa\parentheses{\norm{\p(T)}+\Im}^p}}\sup_{t\in[T,T+h]}\norm{u_1(t)-u_2(t)}_{L^\infty+\Xi}\\[0.3cm]
&\le\mathcal{L}\sup_{t \in[T,T+h]}\norm{u_1(t)-u_2(t)}_{L^\infty+\Xi}. 
\end{align*}
Thanks to an appropriate choice of $ h_4 $, $ \mathcal{L} $ can be proved to be less than $ 1 $.

By setting $ h=\min\left\{h_1,h_2,h_3,h_4\right\} $, we can now apply the Banach principle arguments to declare that the solution to Problem \eqref{Main Problem} has been extended to some larger interval. 
\end{proof}
\begin{theorem}
Let $ T_{\max} $ be the supremum of the set of all $ T>0 $ such that Problem \eqref{Main Problem} has a unique local solution on $ (0,T]. $ Assume that $ G $ satisfies \eqref{Property of J} and $ \p_0 $ belongs to $\Lp\cap{C_0(\R^N)}$. {Then, we can conclude that $ T_{\max}=\infty $ or $ T_{\max}<\infty $ and  $ \limsup\limits_{t\to T_{\max}^-}\norm{\p(t)}_{L^\infty+\Xi}=\infty.$}
\end{theorem}
\begin{proof}
Arguing by contradiction, we assume that $ T_{\max}<\infty $ and there exist a positive constant $ \Im<\infty $ such that 
\begin{align}
\norm{\p(t)}_{L^\infty+\Xi}\le \Im.
\end{align}
Let $ \{t_j\}_{j\in\N}\subset(0,T_{\max})$ satisfy $t_j\xrightarrow{~j\rightarrow\infty~}T_{\max}$. Our goal is to show that $ \{\p(t_j)\}_{j\in\N} $ is a Cauchy sequence. To this end, let us take $ t_a,t_b\in\{t_j\}_{j\in\N}, $   without loss of generality, we suppose that $ t_a<t_b $. Then, we have
\begin{align*}
\norm{\p(t_b)-\p(t_a)}_{L^\infty+\Xi}\le&\norm{\mathbb{Z}_{1,\al}(t_b)\p_{0}-\mathbb{Z}_{1,\al}(t_a)\p_{0}}_{L^\infty+\Xi}\nonumber\\
&+\int_{t_a}^{t_b}\norm{\mathbb{Z}_{2,\al}(T_{\max}-\tau)G(\p(\tau))}_{L^\infty+\Xi}d\tau\\
&+\int_{0}^{t_a}\Bnorm{\sqrbrackets{\mathbb{Z}_{2,\al}(T_{\max}-\tau)-\mathbb{Z}_{2,\al}(t_a-\tau)}G(\p(\tau))}_{L^\infty+\Xi}d\tau\nonumber\\
&\nonumber+\int_{0}^{t_b}\Bnorm{\sqrbrackets{\mathbb{Z}_{2,\al}(t_b-\tau)-\mathbb{Z}_{2,\al}(T_{\max}-\tau)}G(\p(\tau))}_{L^\infty+\Xi}d\tau\nonumber.
\end{align*}
Given $ \varepsilon>0 $, let us consider some large enough constants $j_1,j_2,j_3$ defined through the following steps.
	
\noindent\textbf{\emph{Step 1.}} Thanks to the property of $\mathbb{Z}_{1,\al}(t) $, we can choose a sufficiently large $ j_1 $ such that
\begin{align*}
\norm{\mathbb{Z}_{1,\al}(t_b)\p_{0}-\mathbb{Z}_{1,\al}(t_a)\p_{0}}_{L^\infty+\Xi}<\frac{\varepsilon}{4},\quad\text{for any~}a,b\ge j_1.
\end{align*}
	
\noindent\textbf{\emph{Step 2.}} For the sake of simplicity, let us set
\begin{align*}
\Im^*=L(\log2)^{\frac{-1}{p}}\parentheses{2+\parentheses{\Gamma(m+1)}^{\frac{1}{mp}}}\Im^me^{-\kappa \Im^p}.
\end{align*}
and choose a large $ j_2 $ such that
\begin{align*}
|t_j-T_{\max}|<\frac{\varepsilon\alpha}{8\Im^*},\quad\text{for any~}j\ge j_2.
\end{align*}
Then, similar to the proof of Theorem \ref{LocalExistenceTheoremCase1}, it is easily seen that
\begin{align*}
\int_{t_a}^{t_b}\norm{\mathbb{Z}_{2,\al}(T_{\max}-\tau)G(\p(\tau))}_{L^\infty+\Xi}d\tau&\le\int_{t_a}^{t_b}(T_{\max}-\tau)^{\alpha-1}\norm{G(u(\tau))}_{L^\infty+\Xi}d\tau\nonumber\\
&\le\Im^*\al^{-1}\parentheses{|t_b-T_{\max} |^{\alpha}+|t_a-T_{\max} |^{\alpha}}<\frac{\varepsilon}{4}.
\end{align*}
\textbf{\emph{Step 3.}} The Lebesgue dominated convergence theorem helps us to find a sufficiently large $ j_3 $ such that, for any $ j\ge j_3 $,
\begin{align*}
\int_{0}^{t_j}\norm{\sqrbrackets{\mathbb{Z}_{2,\al}(T_{\max}-\tau)-\mathbb{Z}_{2,\al}(t_j-\tau)}G(\p(\tau))}_{L^\infty+\Xi}d\tau<\frac{\varepsilon}{4}.
\end{align*}
Now choosing $ j_0=\max\{j_1,j_2,j_3\} $ yields that 
\begin{align*}
\norm{\p(t_a)-\p(t_b)}_{L^\infty+\Xi}\le\varepsilon,\quad\text{for every~}a,b\ge j_0.
\end{align*}
Then, there exists a limit of the Cauchy sequence $ \{\p(t_j)\}_{j\in\N} $ in $ \Lep\cap {C_0(\R^N)} $ as $ j $ tends to infinity denoted by $ \p^* $. Hence, we can check that
\begin{align*}
&\norm{\int_{0}^{t_j}\mathbb{Z}_{2,\al}(t_j-\tau)G(\p(\tau))d\tau-\int_{0}^{T_{\max}}\mathbb{Z}_{2,\al}(T_{\max}-\tau)G(\p(\tau))d\tau}_{L^\infty+\Xi}\\
&\le\int_{0}^{t_j}\norm{\sqrbrackets{\mathbb{Z}_{2,\al}(t_j-\tau)-\mathbb{Z}_{2,\al}(T_{\max}-\xi)}G(\p(\tau))}_{L^\infty+\Xi}d\tau+\frac{\Im^*(T_{\max}-t_j)^{\alpha}}{\alpha}\xrightarrow{~j\rightarrow\infty~}0\nonumber.
\end{align*} 
It implies that
\begin{align*}
\p^*:&=\lim\limits_{j\to\infty}\sqrbrackets{\mathbb{Z}_{1,\al}(t_j)\p_{0}+\int_{0}^{t_j}\mathbb{Z}_{2,\al}(t_j-\tau)G(u(\tau))d\tau}\\
&=\mathbb{Z}_{1,\al}(T_{\max})\p_{0}+\int_{0}^{T_{\max}}\mathbb{Z}_{2,\al}(T_{\max}-\tau)G(u(\tau))d\tau.
\end{align*}
This result hepls us to enlarge the solution $ \p $ of the Problem \eqref{Main Problem} over the interval $ [0,T_{\max}] $. Therefore, Lemma \ref{ExtensionLemma} is available to extend $ \p $ to an interval that is larger than $ [0,T_{\max}] $. This fact contradicts our definition of $ T_{\max}. $
\end{proof}

\subsection{Well-posedness results under the second assumption on the initial function}
First of all, let us introduce the important nonlinear estimate for Problem \eqref{Main Problem}. To achieve this aim, we need the following lemma about the Gamma function, that can be found in \cite[Lemma 3.3]{Kawakami Et al.}.  
\begin{lemma}\label{Key-Gamma-estimate}For any $ p,q\ge1 $, there exists a positive constant $ \bar{M} $ such that 
	\begin{align*}
		\sqrbrackets{\Gamma(pq+1)}^{\frac{1}{q}}\le \bar{M}\Gamma(p+1)q^p.
	\end{align*}
\end{lemma}
\begin{lemma}\label{NonlinearEstimate}
Let $  u,v\in\Lep$ and $\mathcal{V}=\max\{\norm{u}_{\Xi},\norm{v}_{\Xi}\}. $ Then, for any $ h\ge p $, the following estimate holds
\begin{align*}
\norm{G(u)-G(v)}_{L^h}\le \mathscr{C}_h\mathcal{V}^{m-1}\sum_{j=0}^{\infty}(3h\kappa\mathcal{V}^p)^j\norm{u-v}_{\Xi}.
\end{align*}
\end{lemma}
\begin{proof}
From \eqref{Property of J}, by using Taylor expansion, for any $u,v\in\Lep, $ one has 
\begin{align}\label{J-J L^h}
\norm{G(u)-G(v)}_{L^h}\le L\sum_{j=0}^{\infty}\frac{\kappa^j}{j!}\norm{u-v}_{L^{3h}}\parentheses{\norm{u}_{L^{3hpj}}^{pj}\norm{u}_{L^{3h(m-1)}}^{m-1}+\norm{v}_{L^{3hpj}}^{pj}\norm{v}_{L^{3h(m-1)}}^{m-1}},
\end{align}	
where we have used the H\"older inequality with $ \frac{1}{h}=\frac{1}{3h}+\frac{1}{3h}+\frac{1}{3h} $. By the embedding stated in Lemma \ref{Embedding exp to Lq}, the estimate \eqref{J-J L^h} becomes
\begin{align}\label{pre-nonlinear-estimate}
\norm{G(u)-G(v)}_{L^h}
\le\frac{2L\mathcal{V}^{m-1}(\Gamma(3hp^{-1}+1))^{\frac{1}{3h}}}{\parentheses{\Gamma(3h(m-1)p^{-1}+1)}^{\frac{-1}{3h}}} \norm{u-v}_{\Xi}\sum_{j=0}^{\infty}\parentheses{\frac{\kappa^j}{j!}\mathcal{V}^{pj}\sqrbrackets{\Gamma(3hj+1)}^{\frac{1}{3h}}}.
\end{align}
It follows from Lemma \ref{Key-Gamma-estimate} that 
\begin{align*}
\sqrbrackets{\Gamma(3hj+1)}^{\frac{1}{3h}}&\le \bar{M}\Gamma\parentheses{j+1}(3h)^{j}=\bar{M}(3h)^jj!.
\end{align*}
Based on the above results, the following approximation is satisfied 
\begin{align}\label{GlobalCase_MainLineareEstimate}
\norm{G(u)-G(v)}_{L^h}\le\frac{2L\mathcal{V}^{m-1}(\Gamma(3hp^{-1}+1))^{\frac{1}{3h}}}{\parentheses{\Gamma(3h(m-1)p^{-1}+1)}^{\frac{-1}{3h}}} \norm{u-v}_{\Xi}\sum_{j=0}^{\infty}(3h\kappa\mathcal{V}^p)^j.
\end{align}
\end{proof}
\subsubsection{Local-in-time solution}
{\begin{theorem}\label{LocalExistenceSmall data}
		Let $ u_0 $ be in $ \Lep $ with sufficiently small data and $ h>\max\left\{p,\frac{N}{4}\right\} $. Then there exists a locally unique mild solution to Problem \eqref{Pro1}, belonging to an open ball centered at the origin with radius $ \varepsilon<(3h\kappa)^{\frac{-1}{p}}. $ 
\end{theorem}}
\begin{proof}
Our proof starts with the observation that
{\begin{align}
		\norm{\mathbb{Z}_{1,\alpha}(t)u_0}_{\Xi}\le\norm{u_0}_{\Xi},\quad t>0,
\end{align}}
where we have used Lemma \ref{Kernel_apply_to_Orlicz}. 
In addition, for any {$ u,v\in L^\infty\parentheses{0,T;\Lep},~T>0 $}, Lemma \ref{NonlinearEstimate} implies for a fixed constant $ h>\max\left\{p,\frac{N}{4}\right\} $ and $ t\in\left(0,T\right) $ that
{\begin{align*}
	&\int_{0}^{t}(t-\tau)^{\al\parentheses{1-\frac{N}{4h}}-1}\norm{G(u(\tau))-G(v(\tau))}_{L^h}d\tau\nonumber\\
	&\le \mathscr{C}_h\sum_{j=0}^{\infty}(3h\kappa)^j\int_{0}^{t}(t-\tau)^{\al\parentheses{1-\frac{N}{4h}}-1}\norm{u(\tau)-v(\tau)}_{\Xi}[\mathcal{V}(\tau)]^{m-1+pj}d\tau\\
	&\le \frac{\mathscr{C}_hT^{\al\parentheses{1-\frac{N}{4h}}}}{\al\parentheses{1-\frac{N}{4h}}}\norm{u-v}_{L^\infty_{T}\Xi}\norm{\mathcal{V}}^{m-1}_{L^\infty_{T}\Xi}\sum_{j=0}^{\infty}(3h\kappa\norm{\mathcal{V}}^{p}_{L^\infty_{T}\Xi})^j\nonumber,
\end{align*}}
here we set $ \mathcal{V}(t)=\max\left\{\norm{u(t)}_{\Xi},\norm{v(t)}_{\Xi}\right\} $. Next, for purpose of using the fixed point principle, let us define the following mapping
\begin{align}\label{G formula}
\mathscr{G}u(t)&=\mathbb{Z}_{1,\al}(t)\p_{0}+\int_{0}^{t}\mathbb{Z}_{2,\al}(t-\tau)G(u(\tau))d\tau,
\end{align}
which maps a closed ball {$ B_{L^\infty\parentheses{0,T;\Lep}}(0,\varepsilon) $} to itself, provided that $ \varepsilon<(3h\kappa)^{\frac{-1}{p}} $ and $ T $ satisfies
\begin{align*}
\frac{\max\left\{\overline{\mathscr{C}_{0,\frac{h}{h-1}}}(\alpha,N)\mathscr{C}_h,\overline{\mathscr{C}_{0,1}}(\alpha,N)\mathscr{C}_p\right\}}{\alpha(\log2)^{\frac{1}{p}}(1-3h\kappa\varepsilon^p)}\sqrbrackets{\parentheses{1-\frac{N}{4h}}^{-1}T^{\al\parentheses{1-\frac{N}{4h}}}+T^{\al}}\varepsilon^{m-1}\le\frac{1}{2}.
\end{align*} 
We first prove the invariance of {$ B_{L^\infty\parentheses{0,T;\Lep}}(0,\varepsilon) $} under the action of $ \mathscr{G} $. Indeed, the Young convolution inequality and Lemma \ref{lemqt} imply
{\begin{align}\label{Lh Estimate for int-term}
		&\norm{\int_{\R^N}\mathbb K_{2,\al} (t-\tau,x-y)\sqrbrackets{G(u(\tau,y))-G(v(\tau,y))}dy}_{L^{\infty}}\\[0.15cm]
		&\nonumber\le\norm{\mathbb K_{2,\al} (t-\tau,x)}_{L^{\frac{h}{h-1}}}\norm{G(u(\tau))-G(v(\tau))}_{L^h}\\
		&\le\overline{\mathscr{C}_{0,\frac{h}{h-1}}}(\alpha,N)(t-\tau)^{\al\parentheses{1-\frac{N}{4h}}-1}\norm{G(u(\tau))-G(v(\tau))}_{L^h}.
\end{align}}
This result leads us to
{\begin{align*}
		\norm{\mathscr{G}u-\mathscr{G}v}_{L^\infty_{T}\Linftyx}\le \overline{\mathscr{C}_{0,\frac{h}{h-1}}}(\alpha,N)\int_{0}^{t}(t-\tau)^{\al\parentheses{1-\frac{N}{4h}}-1}\norm{G(u(\tau))-G(v(\tau))}_{L^h}d\tau.
\end{align*}}	
In \eqref{Lh Estimate for int-term} and the above inequality, if we take $ v\equiv0 $, then from the property of geometric series, the following holds
{\begin{align*}
		\norm{\mathscr{G}u-\mathbb Z_{1,\al}(\cdot)u_{0}}_{L^\infty_{T}\Linftyx}&\le\overline{\mathscr{C}_{0,\frac{h}{h-1}}}(\alpha,N)\mathscr{C}_h\left[\al\parentheses{1-\frac{N}{4h}}\right]^{-1}T^{\al\parentheses{1-\frac{N}{4h}}}\varepsilon^{m}\sqrbrackets{\sum_{j=0}^{\infty}(3h\kappa\varepsilon^p)^j}\nonumber\\
		&=\overline{\mathscr{C}_{0,\frac{h}{h-1}}}(\alpha,N)\mathscr{C}_h\left[\al\parentheses{1-\frac{N}{4h}}\right]^{-1}T^{\al\parentheses{1-\frac{N}{4h}}}\varepsilon^{m}(1-3h\kappa\varepsilon^p)^{-1}.
\end{align*}} 	
Similarly, we also obtain the following estimate for the $ \Lp $ norm
{\begin{align*}
		\norm{\mathscr{G}u-\mathbb Z_{1,\al}u_{0}}_{L^\infty_{T}\Lpx}&\le \overline{\mathscr{C}_{0,1}}(\alpha,N)\mathscr{C}_p\al^{-1}T^{\al}\varepsilon^{m}\sqrbrackets{\sum_{j=0}^{\infty}(3h\kappa\varepsilon^p)^j}=\overline{\mathscr{C}_{0,1}}(\alpha,N)\mathscr{C}_p\al^{-1}T^{\al}\varepsilon^{m}(1-3p\kappa\varepsilon^p)^{-1}.
\end{align*}}	
It follows from the above results that 
{\begin{align*}
		\norm{\mathscr{G}u}_{L^\infty_{T}\Xi}&\le \norm{\mathbb Z_{1,\al}u_{0}}_{L^\infty_{T}\Xi}+(\log2)^{\frac{-1}{p}}\parentheses{\norm{\mathscr{G}u-\mathbb Z_{1,\al}u_{0}}_{L^\infty_{T}\Linftyx}+\norm{\mathscr{G}u-\mathbb Z_{1,\al}u_{0}}_{L^\infty_{T}\Lpx}}\nonumber\\
		&\le \norm{\p_{0}}_{\Xi}+\frac{\max\left\{\overline{\mathscr{C}_{0,\frac{h}{h-1}}}(\alpha,N)\mathscr{C}_h,\overline{\mathscr{C}_{0,1}}(\alpha,N)\mathscr{C}_p\right\}}{\alpha(\log2)^{\frac{1}{p}}(1-3h\kappa\varepsilon^p)}\sqrbrackets{\parentheses{1-\frac{N}{4h}}^{-1}T^{\al\parentheses{1-\frac{N}{4h}}}+T^{\al}}\varepsilon^m\le\varepsilon,
\end{align*}}
where we have used the smallness assumption for the initial data function that {$ 2\norm{u_0}_{\Xi}\le\varepsilon<(3h\kappa)^{\frac{-1}{p}}. $} Hence, the invariance property of $\mathscr{G} $ is ensured.  
Furthermore, by using analogous arguments and the chosen time $ T $, we can easily show that $ \mathscr{G} $ is a strict contraction on $ B_{L^\infty\parentheses{0,T;\Lep}}(0,\varepsilon) $. Then, the Banach principle argument yields that our problem has a unique local-in-time mild solution in $ L^\infty\parentheses{0,T;\Lep} $ and the proof is complete.
\end{proof}

\medskip

\subsubsection{Global-in-time well-posedness results}~
\begin{lemma}{\label{Atienza}}(see \cite[Lemma 8]{Atienza})
Let $ m,n>-1 $ such that $ m+n>-1 $.   Then
\begin{align*}
\sup _{t \in[0, T]} t^{h} \int_{0}^{1} s^{m}(1-s)^n e^{-\mu t(1-s)} \mathrm{d} s\xrightarrow{{\quad\mu\longrightarrow\infty\quad}}0.
\end{align*}
\end{lemma}
\begin{definition}
Let $ X $ be a Banach space. Then, we denote by $ L^{\infty}_{a,b}\parentheses{(0,T];X} $ the subspace of $ L^\infty\left(0,T;X\right) $ such that
\begin{align*}
\sup_{t \in(0, T]}t^ae^{-bt}\norm{\varphi(t)}_X<\infty,\qquad \varphi\in L^{\infty}_{a,b}\parentheses{(0,T];X}.
\end{align*}
for some positive numbers $ a,b. $
\end{definition}
\begin{theorem}
Assume that the initial  function $ \p_{0} $ belongs to $ \0Lep $ and $ m=1 $. Then, Problem \eqref{Main Problem} has a unique solution in $  C\parentheses{(0,T];\Lq}\cap L^{\infty}_{a,b}\parentheses{(0,T];\0Lep}$ for $ a\in(0,1) $ and some $ b_0>0 $. Furthermore, if $ N<4p,a<\min\left\{\frac{1}{2},\frac{\al N}{4p}\right\} $, we have
\begin{align}\label{Stability Inequality}
\norm{u(t)}_{L^p}\le Ct^{-a}e^{b_0t}\norm{u_0}_{\Xi}.
\end{align}
\end{theorem}
\begin{proof}
The proof starts by handling the nonlinear source term on the RHS of the first equation of Problem \eqref{Main Problem}. For every $ u\in L^{\infty}_{a,b}\parentheses{(0,T];\0Lep} $, we can choose two functions $ u_1,u_2 $ such that
\begin{align}\label{split u}
\begin{cases}
\dis u(t)=u_1(t)+u_2(t)\qquad&\forall t\in(0,T],\\
\dis u_1(t)\in C_0^{\infty}(\R^N),&\forall t\in(0,T],\\
\dis\norm{u_2(t)}_{\Xi}<(3h\kappa {2^{p-1}})^{\frac{-1}{p}},&\forall t\in(0,T].
\end{cases}
\end{align}
For brevity, we set
\begin{align*}
\mathscr{C}_T:=\sup_{u\in L^{\infty}_{a,b}\parentheses{(0,T];\0Lep}}\left\{\sup_{t \in(0, T]}\Big\{|u_1(t)|:u_1(t)~\text{satisfies~}\eqref{split u}\Big\}\right\}.
\end{align*}

\noindent Then, for any $h\in\left[p,\infty\right)$, 
\begin{align*}
\norm{G(u(t))-G(v(t))}_{L^h}\le L\norm{|u(t)-v(t)|\parentheses{e^{\kappa|u_1(t)+u_2(t)|^p}+e^{\kappa|v_1(t)+v_2(t)|^p}}}_{L^h}.
\end{align*}
	
\noindent Using the inequality $ (a+b)^q<{2^{p-1}}(a^q+b^q) $ for $ a,b>0,q>1 $, we obtain
\begin{align*}
\norm{G(u(t))-G(v(t))}_{L^h}&\le L\norm{|u(t)-v(t)|\parentheses{e^{\kappa {2^{p-1}}|u^p_1(t)+u^p_2(t)|}+e^{\kappa {2^{p-1}}|v^p_1(t)+v^p_2(t)|}}}_{L^h}\nonumber\\
&\le Le^{\kappa {2^{p-1}}\mathscr{C}_T^p}\norm{|u(t)-v(t)|\parentheses{e^{\kappa {2^{p-1}}|u^p_2(t)|}+e^{\kappa {2^{p-1}}|v^p_2(t)|}}}_{L^h}.
\end{align*}

\noindent Then, it follows from Lemma \ref{NonlinearEstimate} and Lemma \ref{Embedding exp to Lq} that
\begin{align}
\norm{G(u(t))-G(v(t))}_{L^h}&\le C\norm{u(t)-v(t)}_{\Xi}\sum_{j=0}^{\infty}\left(3h\kappa {2^{p-1}}\sqrbrackets{\max_{w\in\{u_2,v_2\}}\left\{\sup_{t \in(0, T]}\norm{w(t)}_{\Xi}\right\}}^p\right)^j.
\end{align}
Then, by \eqref{split u}, the following nonlinear estimate holds
\begin{align}\label{Pseudo_Globally_Lipschitz}
\norm{G(u(t))-G(v(t))}_{L^h}\le C\norm{u(t)-v(t)}_{\Xi}.
\end{align}
Now, we define the following set
\begin{align*}
\mathbf{O}:=\left\{u\in  C\parentheses{(0,T];\Lp}\cap L^{\infty}_{a,b}\parentheses{(0,T];\0Lep}\big|\sup_{t\in(0,T]}t^ae^{-bt}\norm{u(t)}_{\Xi}<\infty\right\}
\end{align*}
and a mapping $ \mathscr{H} $ that maps $ \mathbf{O} $ to itself, formulated by
\begin{align}\label{Formula of H}
\mathscr{H}u(t)=\mathbb{Z}_{1,\al}(t)\p_{0}+\int_{0}^{t}\mathbb{Z}_{2,\al}(t-\tau)G(u(\tau))d\tau.
\end{align}
Since $u_0\in\0Lep$,  Lemma \ref{ContinuityProposition} ensures the continuity on $ (0,T] $ of the first term on the RHS of \eqref{Formula of H}. Furthermore, combining \eqref{Pseudo_Globally_Lipschitz} and the same argument as in Theorem \ref{LocalExistenceTheoremCase1}, we have
\begin{align*}
\mathscr{H}u-\mathbb{Z}_{1,\al}\p_{0}\in C\parentheses{(0,T];\Lp}.
\end{align*} 
In addition, for any $ u\in\mathbf{O} $, $ \mathscr{H}u$ is completely bounded in time.  Notice also that  $ u\in L^{\infty}_{a,b}\parentheses{(0,T];\0Lep} $. Then, the Young convolution inequality and Lemma \ref{lemqt} imply that
\begin{align*}
\norm{\mathscr{H}u(t)-\mathbb{Z}_{1,\al}(t)u_0}_{L^{\infty}}&\le\int_{0}^{t}\norm{\mathbb{Z}_{2,\al}(t-\tau)G(u(\tau))}_{L^{\infty}}d\tau\nonumber\\
&\le\int_{0}^{t}(t-\tau)^{\alpha\parentheses{1-\frac{ N}{4h}}-1}\norm{G(u(\tau))}_{L^h}d\tau,~~ \text{~for~}h>\max\left\{\frac{N}{4},p\right\}.
\end{align*}
By using \eqref{Pseudo_Globally_Lipschitz}, we obtain
\begin{align*}
\norm{\mathscr{H}u(t)-\mathbb{Z}_{1,\al}(t)u_0}_{L^{\infty}}&\le C\sup_{t \in(0, T]}\norm{u(t)}_{\Xi}\parentheses{\int_{0}^{t}(t-\tau)^{\alpha\parentheses{1-\frac{ N}{4h}}-1}d\tau}=\frac{CT^{\alpha\parentheses{1-\frac{ N}{4h}}}}{\alpha\parentheses{1-\frac{ N}{4h}}}\sup_{t \in(0, T]}\norm{u(t)}_{\Xi}.
\end{align*}
Likewise, we also have
\begin{align*}
\norm{\mathscr{H}u(t)-\mathbb{Z}_{1,\al}(t)u_0}_{L^p}\le C\sup_{t \in(0, T]}\norm{u(t)}_{\Xi}\parentheses{\int_{0}^{t}(t-\tau)^{\alpha-1}d\tau}=CT^{\alpha}{\al}^{-1}\sup_{t \in(0, T]}\norm{u(t)}_{\Xi}.
\end{align*}
	
From two  results above, Lemmas \ref{Embedding_Lq_Linfty} and \ref{Kernel_apply_to_Orlicz}, the conclusion $ \mathscr{H}u\in L^{\infty,*}\parentheses{(0,T];\0Lep} $ can be drawn through the following one
\begin{align*}
\norm{\mathscr{H}u(t)}_{\Xi}&\le\norm{\mathbb{Z}_{1,\al}(t)u_0}_{\Xi}+\norm{\mathscr{H}u(t)-\mathbb{Z}_{1,\al}(t)u_0}_{\Xi}\le\norm{u_0}_{\Xi}+C\parentheses{T^{\alpha\parentheses{1-\frac{ N}{4h}}}+T^{\alpha}}\sup_{t \in(0, T]}\norm{u(t)}_{\Xi}.
\end{align*}

We note that $ \mathbf{O} $ is a complete metric space with the metric
\begin{align*}
d_{\mathbf{O}}(u,v):=\sup_{t\in(0,T]}t^ae^{-bt}\norm{u(t)-v(t)}_{\Xi}.
\end{align*}
	
The remaining part of this proof is to show that $ \mathscr{H} $ is a strict contraction on $ \mathbf{O} $ with respect to the above metric. Indeed, for any $ u,v\in\mathbf{O} $, we have
\begin{align}
&t^ae^{-bt}\norm{\mathscr{H}u(t)-\mathscr{H}v(t)}_{L^p+L^\infty}\nn\\
&\le\int_{0}^{t}(t-\tau)^{\alpha-1}\norm{\mathbb{Z}_{2,\al}(t-\tau)\sqrbrackets{G(u(\tau))-G(v(\tau))}}_{L^p+L^\infty}d\tau\nonumber\\
	&\le t^ae^{-bt}\int_{0}^{t}(t-\tau)^{\alpha-1}\norm{G(u(\tau))-G(v(\tau))}_{L^p+L^\infty}d\tau\nn\\
		&\nonumber\le Cd_{\mathbf{O}}(u,v)t^a\int_{0}^{t}\tau^{-a}\sqrbrackets{(t-\tau)^{\alpha-1}+(t-\tau)^{\alpha\parentheses{1-\frac{ N}{4h}}-1}}e^{-b(t-\tau)}d\tau.
	\end{align}
	
	\noindent Using the substitution technique, the latter inequality becomes
	\begin{align*}
		t^ae^{-bt}\norm{\mathscr{H}u(t)-\mathscr{H}v(t)}_{L^p+L^\infty}&\le Cd_{\mathbf{O}}(u,v) t^{\alpha}\int_{0}^{1}\tau^{-a}(1-\tau)^{\alpha-1}e^{-bt(1-\tau)}d\tau\nonumber\\
		&+ Cd_{\mathbf{O}}(u,v) t^{\alpha\parentheses{1-\frac{ N}{4h}}}\int_{0}^{1}\tau^{-a}(1-\tau)^{\alpha\parentheses{1-\frac{ N}{4h}}-1}e^{-bt(1-\tau)}d\tau.
	\end{align*}
Choosing $ 0<a<\alpha<1$ and $ 4h>N $, we  use Lemma~\ref{Atienza} to obtain
	\begin{align*}
		\begin{cases}
			\dis\lim\limits_{b\to \infty}\sup_{t\in(0,T]}t^{\alpha}\int_{0}^{1}\tau^{-a}(1-\tau)^{\alpha-1}e^{-bt(1-r)}d\tau=0,\\
			\dis\lim\limits_{b\to \infty}\sup_{t\in(0,T]}t^{\alpha\parentheses{1-\frac{ N}{4h}}}\int_{0}^{1}\tau^{-a}(1-\tau)^{\alpha\parentheses{1-\frac{ N}{4h}}-1}e^{-bt(1-r)}d\tau=0.
		\end{cases}
	\end{align*}	
Hence, there exists a large number $ b_0 $ such that the following holds
	\begin{align*}
		t^ae^{-b_0t}\norm{\mathscr{H}u(t)-\mathscr{H}v(t)}_{L^p+L^\infty}\le\mathcal{L}(\log2)^{-p}d_{\mathbf{O}}(u,v),
	\end{align*}
	where $ \mathcal{L} $ is a positive constant less than 1. This implies that
	\begin{align*}
		d_{\mathbf{O}}(\mathscr{H}u,\mathscr{H}v)\le\mathcal{L}d_{\mathbf{O}}(u,v).
	\end{align*}
Therefore, $ \mathscr{H} $ has a unique fixed point in $ \mathbf{O} $. Therefore, we can conclude that there exists a unique solution to Problem \eqref{Main Problem}. 
	
Next, let us verify the correctness of the statement \eqref{Stability Inequality} as $ N<4p $. Firstly, if $ t\le1 $, we have
\begin{align*}
t^ae^{-b_0t}\norm{\mathbb{Z}_{1,\al}(t)u_0}_{L^p}\le C\norm{\mathbb{Z}_{1,\al}(t)u_0}_{L^p}\le C\norm{u_0}_{L^p}\le C\norm{u_0}_{\Xi}.
\end{align*}
In addition, if $ t>1 $, by using the Cauchy inequality and Lemma \ref{Kernel_apply_to_Orlicz}, we obtain
\begin{align*}
t^ae^{-b_0t}\norm{\mathbb{Z}_{1,\al}(t)u_0}_{L^p}&\le C\norm{\mathbb{Z}_{1,\al}(t)u_0}_{\Xi}\le Ct^{a-\frac{\al N}{4p}}\left[\log\parentheses{2t^{\frac{-\alpha N}{8}}}\right]^{\frac{-1}{p}}\norm{u_0}_{L^p}\le C\norm{u_0}_{\Xi},
\end{align*} 
where we have used the assumption that $ a\le\frac{\al N}{4p}. $
On the other hand, if $ u\in L^{\infty}\parentheses{(0,T];\0Lep}, $ we can use \eqref{Pseudo_Globally_Lipschitz} to derive the following estimate
\begin{align*}
\norm{G(u(t))}_{L^p}\le C\norm{u(t)}_{\Xi}.
\end{align*}
It follows immediately that 
\begin{align*}
&t^ae^{-b_0t}\int_{0}^{t}(t-\tau)^{\al-1}\norm{\mathbb{Z}_{2,\al}(t)(t-\tau)G(u(\tau))}_{L^p}d\tau \le Ct^ae^{-b_0t}\int_{0}^{t}(t-\tau)^{\al-1}\norm{\mathbb{Z}_{2,\al}(t-\tau)u(\tau)}_{\Xi}d\tau.
\end{align*} 
Thanks to the assumption that $ N<4p $, Lemma \ref{Kernel_apply_to_Orlicz} implies
\begin{align*}
&t^ae^{-b_0t}\norm{u(t)-\mathbb{Z}_{1,\al}(t)u_0}_{L^p}\le Ct^ae^{-b_0t}\int_{0}^{t}(t-\tau)^{\al\parentheses{1-\frac{N}{4p}}-1}\left[\log\parentheses{1+(t-\tau)^{\frac{-\alpha N}{4}}}\right]^{\frac{-1}{p}}\norm{u(\tau)}_{L^p}d\tau.
\end{align*}
Thanks to the H\"older inequality, 
\begin{align*}
t^ae^{-b_0t}\norm{u(t)-\mathbb{Z}_{1,\al}(t)u_0}_{L^p}\le C&\parentheses{t^{2a}\int_{0}^{t}(t-\tau)^{\alpha\parentheses{1-\frac{N}{4p}}-1}\left[\log\parentheses{1+(t-\tau)^{\frac{-\alpha N}{4}}}\right]^{\frac{-2}{p}}\tau^{-2a}e^{-2b_0(t-\tau)}d\tau}^{\frac{1}{2}}
\nonumber\\
\times&\parentheses{\int_{0}^{t}(t-\tau)^{\al\parentheses{1-\frac{N}{4p}}-1}\sqrbrackets{\tau^ae^{-b_0{\tau}}\norm{u({\tau})}_{L^p}}^2d\tau}^{\frac{1}{2}}
	\end{align*}
	
	To deal with the $ \log $ term, let us denote by $ \gamma $ the infimum of the set $ \{z>0: z>2 \log (1+z)\} $. If $t^{\frac{-\alpha N}{4}}>\gamma,$ the results will be covered by the opposite case, for this reason, we only need to consider the case $t^{\frac{-\alpha N}{4}}>\gamma$. As $t^{\frac{-\alpha N}{4}}>\gamma$, thanks to Lemma \ref{Atienza}, the two following claims will be obtained
	
\noindent\textbf{\underline{Claim 1}}. {If $ \tau\le t-\gamma^{\frac{-4}{\alpha N}} $, we have $ (t-\tau)^{\frac{\alpha N}{4}}<\gamma $. This one implies that 
	\begin{align*}
		\log \left(1+(t-\tau)^{\frac{-\alpha N}{4}}\right)>\frac{(t-\tau)^{\frac{-\alpha N}{4}}}{2},\quad\text{~for~} 0<\tau\le t-\gamma^{\frac{-4}{\alpha N}}.
	\end{align*}
	Based on the above inequality, one can derive that}
\begin{align}
&t^{2a}\int_{0}^{t-\gamma^{\frac{-4}{\alpha N}}}(t-\tau)^{\alpha\left(1-\frac{N}{4 p}\right)-1}\left[\log \left(1+(t-\tau)^{\frac{-\alpha N}{4}}\right)\right]^{\frac{-2}{p}}\tau^{-2a}e^{-2b_0(t-\tau)}d\tau\nonumber\\ 
&\le C t^{2a}\int_{0}^{t-\gamma^{\frac{-4}{\alpha N}}}(t-\tau)^{\alpha\left(1+\frac{N}{4 p}\right)-1}\tau^{-2a}e^{-2b_0(t-\tau)}d\tau\\
&\le Ct^{^{\alpha\left(1+\frac{N}{4 p}\right)}}\int_{0}^{1}(1-r)^{\alpha\left(1+\frac{N}{4 p}\right)-1}\tau^{-2a}e^{-2b_0t(1-\tau)}d\tau\le C.\nonumber
\end{align}
	
\noindent\textbf{\underline{Claim 2}}. On the contrary, if $ t-r<\gamma^{\frac{-4}{\alpha N}} $, we have
	\begin{align*}
		&t^{2a}\int_{t-\gamma^{\frac{-4}{\alpha N}}}^{t}(t-\tau)^{\alpha\left(1-\frac{N}{4 p}\right)-1}\left[\log\left(1+(t-\tau)^{\frac{-\alpha N}{4}}\right)\right]^{\frac{-2}{p}}\tau^{-2a}e^{-2b_0(t-\tau)}dr\nonumber\\
		&\le Ct^{2a}[\log (1+\gamma)]^{\frac{-2}{p}} \int_{t-\gamma^{\frac{-4}{\alpha N}}}^{t}(t-\tau)^{\alpha\left(1-\frac{N}{4 p}\right)-1}\tau^{-2a}e^{-2b_0(t-\tau)}d\tau\\
		&\le  t^{^{\alpha\left(1-\frac{N}{4 p}\right)}}\int_{0}^{1}(1-r)^{\alpha\left(1-\frac{N}{4 p}\right)-1}\tau^{-2a}e^{-2b_0t(1-\tau)}d\tau\le C\nonumber.
	\end{align*}
	Combining the above results, the triangle inequality and the inequality $ (a+b)^2\le2(a^2+b^2) $ we deduce
\begin{align*}
t^{2a}e^{-2b_0t}\norm{u(t)}_{L^p}^2\le C\norm{u_0}^2_{\Xi}+C\int_{0}^{t}(t-\tau)^{\al\parentheses{1-\frac{N}{4p}}-1}\sqrbrackets{\tau^ae^{-b_0\tau}\norm{u(\tau)}_{L^p}}^2d\tau.
\end{align*}
	Now, we are in the position to apply the Gr\"onwall inequality to achieve the desired result
	\begin{align*}
		\norm{u(t)}_{L^p}\le Ct^{-a}e^{b_0t}\norm{u_0}_{\Xi}.
	\end{align*}
\end{proof}

\section{ Time-fractional Cahn-Hilliard on the unbounded domain $ \R^N $ }
In this section, we  consider the following time fractional Cahn-Hilliard on $ \R^N $ 
\begin{align}\tag{P2}	\label{Pro}
\begin{cases}
\dis \partial^\al_{0|t}  u(t,x)  + \Delta^2 u(t,x) -  \Delta  F(t,x,u)=0,&\quad  \mbox{in}\quad\R^+\times\R^N, \\ 
\dis u(0,x)= u_0(x), & \quad \mbox{in} \quad \R^N.
\end{cases}
\end{align}

\begin{lemma} (See \cite{Robinson})
	If $1\le b \le p \le d$ and $v \in L^b(\mathbb R^N) \cap L^d(\mathbb R^N)$, then $v \in L^p(\mathbb R^N)$ where
	\begin{equation}
		\|u\|_{L^p} \le \|u\|_{L^b}^\al \|u\|_{L^d}^{1-\al},~~~~\quad \quad \frac{1}{p}= \frac{\al}{b}+\frac{1-\al}{d}
	\end{equation}
	
\end{lemma}

We first consider the following linear problem
\begin{align} 	\label{Pro1}
	\begin{cases}
		D^\al_t  \widehat u (t,\xi)  + |\xi|^4 \widehat u (t,\xi)  = \widehat  {\Delta F } (t,\xi) ,&\quad  \mbox{in} \quad (0,T] \times \mathbb R^N , \\ 
		\widehat u(0,\xi)= \widehat u_0(\xi), & \quad \mbox{in} \quad \mathbb R^N,    
	\end{cases} 
\end{align} 
As in the previous section, by the Duhamel principle, the solution to problem \eqref{Pro1} is given by
\begin{equation} \label{Fu}
	\widehat u (t,\xi)= E_{\al,1} (-t^{\alpha} |\xi|^4) \widehat u_0(\xi)+ \int_0^t (t-s)^{\al-1} E_{\al,\al} (- (t-\tau)^\al  |\xi|^4) \widehat  {\Delta F }  (\tau,\xi) d\xi
\end{equation}

To deal with the source term in the form of $ \Delta H $, we have to do some revisions to the generalized formula of solution. Applying the fact that
$
\frac{d}{dx} \Big(f(x)*g(x) \Big)= \Big( \frac{d}{dx}f(x) \Big)* g(x),
$
we immediately have that
\begin{align*}
	\Delta \mathbb Z_{i,\al}(t) v(x)&=  \Delta \Big( \mathbb K_{i,\al} (t,x) * v(t,x)\Big)=\Big(  \Delta v(t,x) \Big)* \mathbb K_{i,\al} (t,x) = \int_{\mathbb R^N}  \Delta  \mathbb K_{i,\al} (t,x-y) v(t,y)dy ,~~i=1,2.
\end{align*}
Next, we  show that $u$ satisfies  the following equality
\begin{align*}
	u(t,x)= \underbrace{\mathbb Z_{1,\al} (t) u_0(x)}_{\mathscr I_1(t,x)}+ \underbrace{\int_0^t \Delta  \mathbb Z_{2,\al} (t-s) F(u(s,x))ds}_{\mathscr I_2(t,x)}.
\end{align*}
Using \eqref{a11}, we infer  that  the Fourier transform of the first quantity $\mathscr I_1(t,x)$ is given by
\begin{align*}
	\widehat {~~\mathscr I_1~~} (t,\xi)=  \widehat {~~~~\mathbb K_{1,\al}  * u_0~~~~}=\widehat {\mathbb K_{1,\al} } \widehat { u_0} = E_{\al,1} (-t |\xi|^4) \widehat u_0(\xi). 
\end{align*}
It is not difficult to  verify  that the Fourier transformation of the given second quantity $ \mathscr I_2 (t, x) $ is as follows
\begin{align*}
	\widehat {~~\mathscr I_2~~} (t,\xi)&=  \int_{\mathbb R^N } e^{-i <\xi, x>}  \int_0^t \Delta  \mathbb Z_{2,\al} (t-\tau) F(u(\tau,x))d\tau dx\nn\\
	&= \int_0^t \Big( \int_{\mathbb R^N } e^{-i <\xi, x>}   \mathbb K_{2,\al} (t-\tau) \Delta F(u(\tau,x)) dx \Big) d\tau \nn\\
	&= \int_0^t  \mathcal F \Big(  \Big(  \Delta F(u(\tau,x)) \Big)* \mathbb K_{2,\al} (t-\tau,x)\Big) d\tau = \int_0^t \widehat{ ~~\mathbb K_{2,\al}~~} (t-\tau,\xi) \widehat  {\Delta F} (u(\tau,x))  d\tau
\end{align*}
where we have used the formula   
$
\mathcal F (f*g)= \widehat f \widehat g, $
and  
from the following fact $$\widehat{ ~~\mathbb K_{2,\al}~~} (t-\tau,\xi) = (t-\tau)^{\al-1} E_{\al,\al} \big(-(t-\tau)^\al |\xi|^4\big) ,$$ 
we  arrive at the following equality
\[
\widehat {~~\mathscr I_2~~} (t,\xi)= \int_0^t (t-s)^{\al-1} E_{\al,\al} (- (t-s)^\al  |\xi|^4) \widehat  {~~\Delta F~~} (u (s,\xi) ) d\xi.
\]

We establish the local well-posedness of solutions for Problem \eqref{Pro} with small initial data in $\R^N$ by using
Kato’s method (see \cite{Liu}). More precisely,  our main result in this section  can be stated as follows.
\begin{theorem}
Let {$ \alpha>\frac{1}{2} $ and} $ \mathcal{E}>0 $ be a sufficienty small constant. Assume that  
\begin{align*}
{\norm{u_0}_{L^\infty}\le\mathcal{E}\qquad\text{and}\qquad}\max_{|z| \le \mathcal E } \sum_{k=1}^L \Big| D^k F(z)\Big| =\mathcal  A. 
\end{align*}
	Then, Problem \eqref{Pro} has a unique solution $u(t,x)$ on the strip 
	\[
	\mathscr P_{ T_0}= \Big\{ (t,x): 0<t \le  T_0,~~x \in \mathbb R^N  \Big \}
	\]
	such that
	\begin{equation} \label{bound}
		\|u(t)\|_{L^\infty} \le 2 \mathcal E,~~~0\le t \le T_0.
	\end{equation}
	Here  $T_0$ is given by
	\begin{equation} \label{b1}
		T_0 \le  \min \Bigg\{ 1~,~ \bigg(     \frac{\al \Gamma(\al/2+1)} {4\Big( \int_{\mathbb R^N} \Big|  \overline{\mathscr B}_2 (z) \Big| dz \Big) \mathcal A \Gamma \left( \frac{3}{2} \right) \Gamma(\al/2) }  \bigg)^{\frac{2}{\al}}, \bigg(     \frac{\al \Gamma(\al/2+1)} {4\Big( \int_{\mathbb R^N} \Big|  \overline{\mathscr B}_2 (z) \Big| dz \Big) \mathcal A \Gamma \left( \frac{3}{2} \right) }  \bigg)^{\frac{2}{\al}}  \Bigg\}.
	\end{equation}
In addition, for each partition $0< t_1 < t_2<...<t_N <t  \le  T_0$, it holds the following estimate
	\begin{align*}
		{\|D^k u(t)\|_{L^\infty}}\le  (t-t_k)^{-\frac{\al k}{4}} \mathcal Q_k (\mathcal E, \al, t_k-t_1, t-t_k), \quad \quad t_k < t \le T_0,
	\end{align*}
	where $\mathcal Q_k$ is a continuous increasing function of $t-t_k$. 
	
\end{theorem}

\begin{proof}
	Let us 	consider the following integral 
	\begin{align} \label{int1}
		u(t,x)=\int_{\mathbb R^N}  \mathbb K_{1,\al} (t,x-y) u_0(y)dy + \int_0^t \int_{\mathbb R^N}  \mathbb K_{2,\al} (t-\tau,x-y) F(u(\tau,y))d\tau.
	\end{align}
	First, we show that    for  $T_0$   defined by \eqref{b1},   the integro-differential
	equation \eqref{int1} admits a unique continuous solution $u(t, x)$ on the strip  $	\mathscr P_{ T_0}$.  We apply the successive approximations method given in \cite{Liu}.  Let us consider the following sequence
	\begin{align} \label{iteration}
		v_0(t,x)&= u_0(x),\\
		v_n(t,x)&= \int_{\mathbb R^N}\mathbb K_{1,\al} (t,x-y) u_0(y)dy + \int_0^t \int_{\mathbb R^N} \Delta   \mathbb K_{2,\al} (t-\tau,x-y) F(v_{n-1}(\tau,y))dy d\tau\nn\\
		&= \mathbb Z_{1,\al} (t) u_0(x)+ \int_0^t \Delta  \mathbb Z_{2,\al} (t-\tau) F(v_{n-1}(\tau,x))d\tau\nonumber.
	\end{align}
	It is easy to see that $v_n$ is well defined on $[0, \infty) \times \mathbb R^N$. 
	For $n=0$, we have immediately that
	$
	\|v_0\|_{L^\infty} = \|u_0\|_{L^\infty} \le  \mathcal E. 
	$
Thus, we apply inequality \eqref{K11} for $p=1$ to derive
\begin{align}  \label{b111}
\Big\|\mathbb K_{1,\al} (t)\Big\|_{L^1(\mathbb R^N)} \le   \frac{ \Big( \int_{\mathbb R^N} \Big|  \overline{\mathscr B}_0 (z) \Big| dz \Big)   \Gamma (1) }{ \Gamma (1) } {\le} 1,
\end{align}
where  we have used the fact that
	$
		\int_{\mathbb R^N}  \Big|  \overline{\mathscr B}_0 (z) \Big| dz= \int_{\mathbb R^N}  \Big|  \int_{\mathbb R^N} e^{i z\vartheta  } e^{-|\vartheta|^4} d\vartheta \Big| dz=1.
	$
	Now, we apply induction method to show the following inequality 
	\begin{equation} \label{qt11}
		\opnorm{v_j}_{\mathscr P_{ T_0}}=	\sup_{(t,x) \in {\mathscr P_{ T_0}}} |v_j(t,x)| \le 2 \mathcal E,~~j \ge 1.
	\end{equation}
For $n=1$, we recall the result based on the Young convolution inequality and \eqref{b111} as follows
\begin{align}\label{b11}
\Bigg\|\int_{\mathbb R^N}   \mathbb K_{1,\al} (t,\cdot-y) u_0(y)dy\Bigg\|_{L^\infty}&=\Big\|  \mathbb K_{1,\al} (t) * u_0 \Big\|_{L^\infty}\le \|  \mathbb K_{1,\al} (t) \|_{L^1(\mathbb R^N)}  \|   u_0 \|_{L^\infty} \le  \|   u_0 \|_{L^\infty} \le  \mathcal E. 
\end{align}
	Let us continue to verify that
	\begin{align*} 
		\Bigg\|\int_0^t \Delta  \mathbb Z_{2,\al} (t-\tau) F(v_{0}(\tau,x))d\tau \Bigg\|_{L^\infty}&=\Bigg\|\int_0^t  \int_{\mathbb R^N}  \Delta  \mathbb K_{2,\al} (t-\tau,x-y) F(v_{0}(\tau,y)) dy d\tau \Bigg\|_{L^\infty} \nn\\
		&\le \int_0^t  \Bigg\| \int_{\mathbb R^N}  \Delta  \mathbb K_{2,\al} (t-\tau,x-y) F(v_{0}(\tau,y)) dy \Bigg\|_{L^\infty} d\tau\nn\\
		& \le \int_0^t  \|  \Delta  \mathbb K_{2,\al} (t-\tau,x)  \|_{L^1(\mathbb R^N)}  \|  F(v_{0}(\tau,x)) \|_{L^\infty}   d\tau. 
	\end{align*} 
	By taking $p = 1$ and $k=2$ into  inequality  \eqref{K22222}, we obtain   the following bound immediately
	\begin{align} \label{K22}
		\Big\|\Delta \mathbb K_{2,\al} (t,x) \Big\|_{L^1(\mathbb R^N)}&=	\Big\| D^2 \mathbb K_{2,\al} (t,x) \Big\|_{L^1(\mathbb R^N)}\le  	 \frac{ \Big( \int_{\mathbb R^N} \Big|  \overline{\mathscr B}_2 (z) \Big| dz \Big) \Gamma \left( \frac{3}{2} \right)}{\Gamma \left( 1+\frac{\al}{2} \right)}   t^{\frac{\al}{2}-1} ,
	\end{align}
	which combined  with the condition
	\begin{equation} \label{condF}
		\|  F(v_{0}(s,x)) \|_{L^\infty} \le \mathcal A \sup_{(t,x) \in \mathscr S_0} |v_0(t,x)|, 
	\end{equation}
	imply
\begin{align} \label{b12}
\Bigg\|\int_0^t \Delta  \mathbb Z_{2,\al} (t-s) F(v_{0}(s,x))ds\Bigg\|_{L^\infty}&\le     \frac{ \Big( \int_{\mathbb R^N} \Big|  \overline{\mathscr B}_2 (z) \Big| dz \Big) \mathcal A \Gamma \left( \frac{3}{2} \right)}{\Gamma \left( 1+\frac{\al}{2} \right)}  \sup_{(t,x) \in \mathscr S_0} |v_0(t,x)| \int_0^t (t-s)^{\al/2-1}     ds \nn\\
& \le  \frac{ \Big( \int_{\mathbb R^N} \Big|  \overline{\mathscr B}_2 (z) \Big| dz \Big) \mathcal A \Gamma \left( \frac{3}{2} \right)}{\Gamma \left( 1+\frac{\al}{2} \right)}   \frac{2t^{\al/2}}{\al} 2 \mathcal E  \nn\\
&\le    \frac{ 4\Big( \int_{\mathbb R^N} \Big|  \overline{\mathscr B}_2 (z) \Big| dz \Big) \mathcal A \Gamma \left( \frac{3}{2} \right) |T_0|^{\al/2} \mathcal E}{\al \Gamma \left( 1+\frac{\al}{2} \right)}    \le   \mathcal E.
\end{align}
	Estimates \eqref{b11} and \eqref{b12} yield that 
	\begin{align*}
		\sup_{(t,x) \in \mathscr S_0} |v_1(t,x)|  &\le \Big\|\mathbb T_{1,\al} (t) u_0\Big\|_{L^\infty}+\Bigg\|\int_0^t \Delta  \mathbb T_{2,\al} (t-\tau) F(v_{0}(\tau,x))ds\Bigg\|_{L^\infty}\le  \mathcal E+ \mathcal E=  2\mathcal E,
	\end{align*}
	where $T_0$ is given by \eqref{b1}. 
	Let us assume that 
	$
	\sup_{(t,x) \in \mathscr S_0} |v_n(t,x)|  \le  2\mathcal E,~~n \ge 1.
	$
	Then for any $(t,x) \in \mathscr S_0$, it follows
	\begin{align*}
		|v_{n+1}(t,x)|  &\le \Bigg\| \int_{\mathbb R^N}\mathbb K_{1,\al} (t,x-y) u_0(y)dy\Bigg\|_{L^\infty}+ \Bigg\|\int_0^t  \int_{\mathbb R^N}  \Delta  \mathbb K_{2,\al} (t,x-y) F(v_{n}(\tau,y)) dy d\tau  \Bigg\|_{L^\infty}\nn\\
		&\le \big\|  \mathbb K_{1,\al} (t,x) \big\|_{L^1(\mathbb R^N)}  \|   u_0 \|_{L^\infty}+ \int_0^t  \big\|  \Delta  \mathbb K_{2,\al} (t,x)  \big\|_{L^1(\mathbb R^N)}  \big\|  F(v_{n}(\tau,x)) \big\|_{L^\infty}   d\tau \nn\\
		& \le \mathcal E+ \frac{ \Big( \int_{\mathbb R^N} \Big|  \overline{\mathscr B}_2 (z) \Big| dz \Big) \mathcal A \Gamma \left( \frac{3}{2} \right)}{\Gamma \left( 1+\frac{\al}{2} \right)}  \sup_{(t,x) \in \mathscr S_0} |v_n(t,x)| \int_0^t (t-\tau)^{\al/2-1}     d\tau  \nn\\
		&\le \mathcal E+ \frac{ 4\Big( \int_{\mathbb R^N} \Big|  \overline{\mathscr B}_2 (z) \Big| dz \Big) \mathcal A \Gamma \left( \frac{3}{2} \right) |T_0|^{\al/2} \mathcal E}{\al \Gamma \left( 1+\frac{\al}{2} \right)} \le 2  \mathcal E,
	\end{align*}
	where it follows from \eqref{b1} that  $
	\frac{ 4\Big( \int_{\mathbb R^N} \Big|  \overline{\mathscr B}_2 (z) \Big| dz \Big) \mathcal A \Gamma \left( \frac{3}{2} \right) |T_0|^{\al/2} }{\al \Gamma \left( 1+\frac{\al}{2} \right)}   \le 1.$
	The latter inequality is true for $j=n$ and, by induction, we deduce that \eqref{qt11} holds for  any $j \ge 0$. \\
	In the following,  we  show that  the following estimate holds for $j \ge 0$
	\begin{align} \label{important1}
		\sup_{x \in \mathbb R^N}	|v_{j+1}(t,x)- v_{j}(t,x)|   \le \frac{  (\mathcal C_\al t^{\frac{\al}{2}})^{j+1}}{\Gamma (\frac{\al j}{2}+\frac{\al}{2}+1)}  ,
	\end{align}
	where $ \mathcal C_\al$ is given by 
	\begin{equation} \label{calpha}
		\mathcal C_\al = 4 \Big( \int_{\mathbb R^N} \Big|  \overline{\mathscr B}_2 (z) \Big| dz \Big) \mathcal A \Gamma \left( \frac{3}{2} \right) \max \left( \frac{\mathcal E}{\al},~~\frac{\Gamma(\al/2)}{2\Gamma(\al/2+1)} \right). 
	\end{equation}
	Indeed, for $j=0$, using \eqref{K22} and \eqref{condF}, we find that
	\begin{align*}
		\sup_{x \in \mathbb R^N}	\Big|v_{1}(t,x)- v_{0}(t,x)\Big |  &\le \Bigg\|\int_0^t \Delta  \mathbb K_{2,\al} (t-\tau) F(v_{0}(\tau,x))ds\Bigg\|_{L^\infty}\nn\\
		&\le \int_0^t  \|  \Delta  \mathbb K_{2,\al} (t-\tau,x)  \|_{L^1(\mathbb R^N)}  \|  F(v_{0}(\tau,x)) \|_{L^\infty}   d\tau \nn\\
		&\le  \frac{ 2 \mathcal E\Big( \int_{\mathbb R^N} \Big|  \overline{\mathscr B}_2 (z) \Big| dz \Big) \mathcal A \Gamma \left( \frac{3}{2} \right)}{\Gamma \left( 1+\frac{\al}{2} \right)}  \int_0^t (t-\tau)^{\al/2-1}     d\tau\nn\\
		& =  \frac{ 4 \mathcal E t^{\al/2}\Big( \int_{\mathbb R^N} \Big|  \overline{\mathscr B}_2 (z) \Big| dz \Big) \mathcal A \Gamma \left( \frac{3}{2} \right)}{\al \Gamma \left( 1+\frac{\al}{2} \right)}   \le \frac{t^{\al/2} \mathcal C_\al }{\Gamma(\al/2+1)} ,
	\end{align*}
	which means that \eqref{important1} holds for $j=0$. Suppose that \eqref{important1}  holds for
	$j \le n-1$,  where $n \ge 2$ is a positive integer. From \eqref{iteration}, we find that
	
	\begin{align} \label{g1}
		\sup_{x \in \mathbb R^N} \Big|v_{n}(t,x)- v_{n-1}(t,x)\Big| &\le \Bigg\|\int_0^t \Delta  \mathbb K_{2,\al} (t-\tau) \Big( F(v_{n-1}(\tau,x)- F(v_{n-2}(s,x) \Big)  d\tau \Bigg\|_{L^\infty}\nn\\
		&\le \int_0^t  \Big\|  \Delta  \mathbb K_{2,\al} (t-\tau,x)  \Big\|_{L^1(\mathbb R^N)}  \Big\| F(v_{n-1}(\tau,x)- F(v_{n-2}(\tau,x)\Big\|_{L^\infty}   d\tau \nn\\
		&  \le  \frac{ 2 \mathcal E\Big( \int_{\mathbb R^N} \Big|  \overline{\mathscr B}_2 (z) \Big| dz \Big) \mathcal A \Gamma \left( \frac{3}{2} \right)}{\Gamma \left( 1+\frac{\al}{2} \right)}      \int_0^t (t-\tau)^{\al/2-1}  \|v_{n-1}(\tau,.)- v_{n-2}(\tau,.)\|_{L^\infty}   d\tau \nn\\
		&    \le \frac{ 2 \mathcal E\Big( \int_{\mathbb R^N} \Big|  \overline{\mathscr B}_2 (z) \Big| dz \Big) \mathcal A \Gamma \left( \frac{3}{2} \right)}{\Gamma \left( 1+\frac{\al}{2} \right)}      \int_0^t  \frac{ (t-\tau)^{\al/2-1} (\mathcal C_\al t^{\frac{\al}{2}})^{n-1}}{\Gamma (\frac{\al n}{2}-\frac{\al}{2}+1)}   d\tau.
	\end{align}
	Noting that $\Gamma(\al) \le 1$,  we observe the following:
	\begin{align*}
		\underline{ \text  {\it The RHS of \eqref{g1}}} &= \frac{ 2 \mathcal E\Big( \int_{\mathbb R^N} \Big|  \overline{\mathscr B}_2 (z) \Big| dz \Big) \mathcal A \Gamma \left( \frac{3}{2} \right)}{\Gamma \left( 1+\frac{\al}{2} \right)}   \frac{  (\mathcal C_\al)^{n-1}}{\Gamma (\frac{\al n}{2}-\frac{\al}{2}+1)}   \int_0^t (t-\tau)^{\al/2-1}  \tau^{\frac{\al n-\al}{2}} d\tau \nn\\
		&= \frac{ 2 \mathcal E\Big( \int_{\mathbb R^N} \Big|  \overline{\mathscr B}_2 (z) \Big| dz \Big) \mathcal A \Gamma \left( \frac{3}{2} \right)}{\Gamma \left( 1+\frac{\al}{2} \right)}   \frac{  (\mathcal C_\al)^{n-1}}{\Gamma (\frac{\al n}{2}-\frac{\al}{2}+1)}  t^{\frac{\al n}{2}} \frac{\Gamma(\al/2)\Gamma(\frac{\al n-\al }{2}+1)}{\Gamma(\frac{\al n}{2}+1)}\nn\\
		& \le \frac{ 2 \Gamma(\al/2) \mathcal E\Big( \int_{\mathbb R^N} \Big|  \overline{\mathscr B}_2 (z) \Big| dz \Big) \mathcal A \Gamma \left( \frac{3}{2} \right)}{\Gamma \left( 1+\frac{\al}{2} \right)}   \frac{ (\mathcal C_\al )^{n-1}}{\Gamma(\frac{\al n}{2}+1)}   t^{\frac{\al n}{2}} ~~\le~~  \frac{ (\mathcal C_\al t^{\al/2} )^{n}}{\Gamma(\frac{\al n}{2}+1)} ,
	\end{align*}
	where in the last inequality, we have used from \eqref{calpha} that  $$\frac{ 2 \Gamma(\al/2) \mathcal E\Big( \int_{\mathbb R^N} \Big|  \overline{\mathscr B}_2 (z) \Big| dz \Big) \mathcal A \Gamma \left( \frac{3}{2} \right)}{\Gamma \left( 1+\frac{\al}{2} \right)}    \le \mathcal C_\al.$$
	By the induction method, we derive that the estimate
	\eqref{important1} holds for any $j\ge 1.$
	It  follows from  \eqref{important1} that
	\begin{align} \label{hold}
		\opnorm{v_{j+1}-v_j}_{\mathscr P_{ T_0}} &= \sup_{(t,x) \in \mathscr P_{ T_0} }	\Big|v_{j+1}(t,x)- v_{j}(t,x)\Big|\le  \sup_{0 \le t \le  T_0}\frac{  (\mathcal C_\al t^{\frac{\al}{2}})^{j+1}}{\Gamma (\frac{\al j}{2}+\frac{\al}{2}+1)} =\frac{  (\mathcal C_\al T_0^{\frac{\al}{2}})^{j+1}}{\Gamma (\frac{\al j}{2}+\frac{\al}{2}+1)}.
	\end{align}
	Since \eqref{hold}, we deduce that for $m> n$
	\begin{align} \label{c1}
		\opnorm{v_{m}-v_n}_{\mathscr P_{ T_0}} &= \sup_{(t,x) \in \mathscr P_{ T_0} }	\Big|v_{m}(t,x)- v_{n}(t,x)\Big|\le \sum_{j=n}^m \big\|v_{j+1}(t,x)- v_{j}(t,x)\big\|_{L^\infty}\le \sum_{j=n}^m  \frac{  (\mathcal C_\al T_0^{\frac{\al}{2}})^{j+1}}{\Gamma (\frac{\al j}{2}+\frac{\al}{2}+1)} .  
	\end{align}
	In the next step,	we claim that the infinite sum $\sum_{j=1}^\infty \frac{  (\mathcal C_\al T_0^{\frac{\al}{2}})^{j+1}}{\Gamma (\frac{\al j}{2}+\frac{\al}{2}+1)}   $ is convergent. 
	Due to the definition of $T_0$ as in \eqref{b1}, we find  that
	$
	T_0^{\al/2} \le   \frac{\al \Gamma(\al/2+1)} {4\Big( \int_{\mathbb R^N} \Big|  \overline{\mathscr B}_2 (z) \Big| dz \Big) \mathcal A \Gamma \left( \frac{3}{2}  \right)\Gamma(\al/2) } .
	$
	By relying on \eqref{calpha}, we can easily achieve that
	\begin{align*}
		(\mathcal C_\al T_0^{\al/2})^{j+1} \le \left( 4 \Big( \int_{\mathbb R^N} \Big|  \overline{\mathscr B}_2 (z) \Big| dz \Big) \mathcal A \Gamma \left( \frac{3}{2} \right)  \frac{\Gamma(\al/2)}{2\Gamma(\al/2+1)}  T_0^{\al/2} \right)^{j+1} \le \left( \frac{\alpha}{2} \right)^{j+1}.
	\end{align*}
	Since $\al>1/2$, we know that $\al j+\al+1 >2$ for $j \ge 1$. Due to the fact that the function $\Gamma(z)$ is  increasing  for $ z>2$, we find  that $\Gamma (\al j+\al+1) > \Gamma (2\al+1)$. It follows from \eqref{c1} that for $m > n \ge \overline{M}$
	\begin{align*}
		\big\|v_{m}(t,x)- v_{n}(t,x)\big\|_{L^\infty}  &\le \frac{1}{\Gamma (2\al+1)} \sum_{j=n}^m \left( \frac{\alpha}{2} \right)^{j+1}\le \frac{1}{\Gamma (2\al+1)} \sum_{j=\overline{M}}^\infty \left( \frac{\alpha}{2} \right)^{j+1} \\
		&\le \frac{2}{(2-\al) \Gamma (2\al+1)}\left(\frac{\alpha}{2} \right)^{\overline{M}+1}.
	\end{align*}
	Now, given any $\ep>0$, we can pick $\overline{M}$,  depending on $\ep$, such that
	$
	\frac{2}{(2-\al) \Gamma (2\al+1)}\left(\frac{\alpha}{2} \right)^{\overline{M}+1} < \ep.
	$
	Some of above observations allow us to conclude that the sequence $\{v_n\}$ is a Cauchy one in the space $L^\infty(\mathbb R^N)$. Therefore, there exists a function $v(t,x)$ which is the limitation of the sequence $\{v_n\}$ on the strip $	\mathscr P_{ T_0}$. It is obvious to see that $v$ is a continuous solution of the integral equation \eqref{int1} on the strip $	\mathscr P_{ T_0}$.
	Next, we examine the regularity  of the solution  $u$. We only need to derive the following estimation. 
	For each $1\le k \le L,~~n \ge 1$, there exists $\mathscr Q_k$ which is a continuous increasing function  of $t-t_k$ such that,
	for each $0< t_1 < t_2<...<t_N <t  \le  T_0$,  the following estimate holds true
	\begin{align*}
		\|D^k v_n(t,x)\|_{L^\infty} \le (t-t_k)^{-\frac{\al k}{4}} \mathcal Q_k (\mathcal E, \al, t_k-t_1, t-t_k), \quad \quad t_k < t \le T_0. 
	\end{align*}
	From the formula \eqref{iteration}, we find that
	\begin{align} \label{iteration1}
		D v_n(t,x)= D \mathbb Z_{1,\al} (t- t_1) v_n (t_1, x)+ \int_{t_1}^t D^3 \mathbb Z_{2,\al} (t-\tau) F(v_{n-1}(\tau,x))d\tau,`~~t_1 \le t \le T_0.
	\end{align}
	This implies that
	\begin{align} \label{c11}
	\sup_{x \in \mathbb R^N }	\Big| D v_{n}(t,x)\Big|&\le \sup_{x \in \mathbb R^N }	\Big| D \mathbb Z_{1,\al} (t- t_1) v_n (t_1, x) \Big| + \sup_{x \in \mathbb R^N }	\Bigg| \int_{t_1}^t D^3 \mathbb Z_{2,\al} (t-\tau) F(v_{n-1}(\tau,x))d\tau \Bigg| \nn\\
		&=\underbrace{ \sup_{x \in \mathbb R^N } \Big| D \Big(\mathbb K_{1,\al} (t-t_1,x) * v_n(t_1,x) \Big)  \Big|}_{(I)}+ \underbrace{  \sup_{x \in \mathbb R^N } \int_{t_1}^t \left| D^3 \mathbb Z_{2,\al} (t-\tau) F(v_{n-1}(\tau,x))  \right| d\tau}_{(II)}.
	\end{align}
	Using the fact that $
	\frac{d}{dx} \Big(f(x)*g(x) \Big)= \Big( \frac{d}{dx}f(x) \Big)* g(x)
	$ and thanks to Lemma \eqref{lemqt},  the term $(I)$ is bounded by 
	\begin{align} \label{g22}
		(I)&= \sup_{x \in \mathbb R^N }  \Big| D \Big(\mathbb K_{1,\al} (t-t_1,x) * v_n(t_1,x) \Big)  \Big| =  \Big\|  D \mathbb K_{1,\al} (t-t_1, .) * v_n(t_1,.) \Big\|_{L^\infty} \nn\\
		&\le \|   D \mathbb K_{1,\al} (t-t_1, .)  \|_{L^1(\mathbb R^N)}  \|   v_n(t_1,.) \|_{L^\infty}~~\le~~ 2 \mathcal  E \mathscr C_{1,1} (\al, N) (t-t_1)^{-\frac{\al }{4}}  
	\end{align}
	Using  \eqref{bound} and the second part of  Lemma \eqref{lemqt} with $p=1, k=3$,  we find that the  term $(II)$ is bounded by
	\begin{align} \label{g2}
		(II) &\le  \int_{t_1}^t  \left\| D^3 \mathbb Z_{2,\al} (t-\tau) F(v_{n-1}(\tau,x))  \right\|_{L^\infty} d\tau\nn\\
		&\le   \int_{t_1}^t   \big\|  D^3 \mathbb K_{2,\al} (t-\tau,.)  \big\|_{L^1(\mathbb R^N)}  \|  F(v_{n-1}(\tau,.)) \|_{L^\infty}   d\tau \nn\\
		&\le \mathcal A \sup_{(t,x) \in [0, T_0] \times \mathbb R^N }    |v_{n-1}(t,x)| \int_{t_1}^t   \Big\|  D^3 \mathbb K_{2,\al} (t-\tau,.)  \Big\|_{L^1(\mathbb R^N)}    d\tau\nn\\
		&\le   2 \mathcal E \mathcal A  \overline{ \mathscr C_{3,1}} (\al, N)  \int_{t_1}^t  (t-\tau)^{\al-\frac{\al N}{4}-1+\frac{\al N}{4}-\frac{3\al }{4}}  d\tau~=~ \frac{ 8\mathcal E \mathcal A  \overline{ \mathscr C_{3,1}} (\al, N) }{\al} (t-t_1)^{\frac{\al}{4}}. 
	\end{align}
	Combining \eqref{c11}, \eqref{g22} and \eqref{g2}, we deduce that there exists $\mathscr Q_1$ which is a continuous increasing function  of $t-t_1$ such that
	\begin{align}  \label{j1}
		\sup_{x \in \mathbb R^N }	\Big| D v_{n}(t,x)\Big| \le (t-t_1)^{\frac{-\al}{4}} \mathscr Q_1 (\mathcal E, \al,  t-t_1). 
	\end{align} 
	From the formula \eqref{iteration}, we find that
	\begin{align} \label{iteration2}
		D^2 v_n(t,x)= D^2 \mathbb Z_{1,\al} (t- t_2) v_n (t_1, x)+ \int_{t_2}^t D^4 \mathbb Z_{2,\al} (t-\tau) F(v_{n-1}(\tau,x))d\tau,`~~t_2 \le t \le T_0.
	\end{align}	
	By a similar argument as above, we find that
	\begin{align} \label{c111}
		\sup_{x \in \mathbb R^N }	\Big| D^2 v_{n}(t,x)\Big|&\le  \sup_{x \in \mathbb R^N }	\Big| D^2 \mathbb Z_{1,\al} (t- t_2) v_n (t_2, x) \Big| + \sup_{x \in \mathbb R^N }	\Bigg| \int_{t_2}^t D^4 \mathbb Z_{2,\al} (t-\tau) F(v_{n-1}(\tau,x))d\tau \Bigg| \nn\\
		&\nn\le   \Big\|   D^2 \mathbb K_{1,\al} (t-t_2, .)  \Big\|_{L^1(\mathbb R^N)}  \|   v_n(t_2,.) \|_{L^\infty} \nn\\
		&+  \int_{t_2}^t   \big\|  D^3 \mathbb K_{2,\al} (t-\tau,.)  \big\|_{L^1(\mathbb R^N)}  \|  DF(v_{n-1}(\tau,.)) \|_{L^\infty}   d\tau \nn\\
		&\nn\le  2 \mathcal  E \mathscr C_{2,1} (\al, N) (t-t_2)^{-\frac{\al }{2}}  \nn\\
		&+  (12 \mathcal E^2+1) \overline{ \mathscr C_{3,1}} (\al, N) \int_{t_2}^t \mathscr Q_1 (\mathcal E, \al,  \tau -t_1)  (\tau -t_1)^{\frac{-\al}{4}} (t-\tau)^{\frac{\al}{4}-1} d\tau,
	\end{align}
	where it follows from \eqref{j1} that
	\begin{align*}
		\big\|  DF(v_{n-1}(\tau,.)) \big\|_{L^\infty} &\le \|3 v_{n-1}^2-1\|_{L^\infty} \|  D(v_{n-1}(\tau,.)) \|_{L^\infty}\le (12 \mathcal E^2+1)  (\tau -t_1)^{\frac{-\al}{4}} \mathscr Q_1 (\mathcal E, \al,  \tau -t_1).
	\end{align*}
	Now, we handle the integral term on the right hand side of the expression \eqref{c111}. It is noted that $(\tau -t_1)^{\frac{-\al}{4}}  \le (t_2 -t_1)^{\frac{-\al}{4}} $ for any $\tau \ge t_2$, and we find that
	\begin{align*}
		\int_{t_2}^t \mathscr Q_1 (\mathcal E, \al,  \tau -t_1)  (\tau -t_1)^{\frac{-\al}{4}} (t-\tau)^{\frac{\al}{4}-1} d\tau \le \mathscr Q_1 (\mathcal E, \al,  t_2 -t_1)(t_2 -t_1)^{\frac{-\al}{4}} \int_{t_2}^t (t-\tau)^{\frac{\al}{4}-1} d\tau
	\end{align*}
	where we recall that $\mathscr Q_1$  is a continuous increasing function  of $t-t_1$. Therefore, it follows from the latter above estimate that
	\begin{align} \label{c112}
		\int_{t_2}^t \mathscr Q_1 (\mathcal E, \al,  \tau -t_1)  (\tau -t_1)^{\frac{-\al}{4}} (t-\tau)^{\frac{\al}{4}-1} d\tau \le \frac{4\mathscr Q_1 (\mathcal E, \al,  t_2 -t_1)(t_2 -t_1)^{\frac{-\al}{4}} }{\al } (t-t_2)^{\al/4}.
	\end{align}
	Combining \eqref{c111} and \eqref{c112}, we arrive at
	\begin{align*}
		\sup_{x \in \mathbb R^N }	\Big| D^2 v_{n}(t,x)\Big| \le \mathscr Q_2 (\mathcal E, \al,  t_2 -t_1, t-t_2)  (t-t_2)^{-\frac{\al }{2}},
	\end{align*}
	where
	\begin{equation*}
		\mathscr Q_2 (\mathcal E, \al,  t_2 -t_1, t-t_2) =2 \mathcal  E \mathscr C_{2,1} (\al, N)+ \frac{4\mathscr Q_1 (\mathcal E, \al,  t_2 -t_1)(t_2 -t_1)^{\frac{-\al}{4}} }{\al } (t-t_2)^{3\al/4}.
	\end{equation*}
	It is easy to verify that $\mathscr Q_2$ as above is a continuous increasing function  of $t-t_2$.
	By a similar way as above, we can verify that $\mathscr Q_k$ as above is a continuous increasing function  of $t-t_k$ for any $k $ is a natural number such that $k \ge 2$. This completes our proof. 
\end{proof}

\section{Appendix}
\begin{definition}\label{Def of Gamma,Wright}
	Let $ \alpha $ be a complex number  whose real part is positive. The Gamma function can be formulated as
$
		\Gamma(\alpha)=\int_{0}^{\infty}\frac{x^{\alpha-1}}{e^x}dx,
$
	and the M-Wright type function $\mathcal{M}_{\alpha}$ is given by
$
		\mathcal{M}_{\alpha}(z)=\sumuse\frac{(-z)^m}{m!\Gamma\sqrbrackets{-\alpha m+(1-\alpha)}}.
$
\end{definition}
\begin{lemma}{(See \cite[Proposition 2]{Planas} or \cite[Appendix F]{Mainardi's Book})}\label{lem21}
Let $\al \in (0,1)$ and $\theta>-1$. Then, the following properties holds
\begin{align}
\mathcal M_\al(\nu) \ge 0, ~~ \forall \nu \ge 0, \quad \textrm{and} \quad \int_0^\infty \nu^\theta \mathcal M_\al(\nu) d\nu = \frac{\Gamma(\theta+1)}{\Gamma(\theta\alpha+1)}, ~~ \forall \theta>-1.\label{PhiProp}
\end{align}
{\begin{proposition}
		Let $ a,b>0 $ and $ q\ge1 $. Then the following inequality holds
		\begin{align*}
			(a+b)^p\le 2^{q-1}\left(a^q+b^q\right).
		\end{align*}
	\end{proposition}
		\begin{proof}
			From the fact that $ q(q-1)x^{q-2}\ge0 $ for any $ x>0 $ and $ q\ge1 $, we assert that the one variable function $ f(x)=x^q,~q>0 $ is a convex function. It follows that
			\begin{align*}
				f\left(\frac{\sum_{k=1}^{M}a_k}{M}\right)\le\frac{\sum_{k=1}^{M}f(a_k)}{M}.
			\end{align*}
			This one gives us the desired result.
		\end{proof}}
\end{lemma}
\begin{lemma}(Young's convolution inequality) Let $ p,q,r\in[1,\infty] $ such that
\begin{align*}
1+\frac{1}{r}=\frac{1}{p}+\frac{1}{q}.
\end{align*} 
Then, the inequality 
$
\norm{u*v}_{L^r}\le\norm{u}_{L^p}\norm{v}_{L^q}
$
holds for every $ u\in\Lp $ and $ v\in\Lq $.
\end{lemma}

\begin{lemma}
(Fractional Gr\"onwall inequality) Let $ m,n$ be positive constants and $ \zeta\in(0,1). $ Assume that function $ \p\in L^{\infty,*}(0,T] $ satisfies the following inequality
\begin{align*}
\p(t)\le m+n\int_{0}^{t}(t-\tau)^{\zeta-1}\p(r)dr,\quad\text{for all~}t\in(0,T],
\end{align*} 
then, the result below is satisfied
\begin{align*}
\p(t)\le mE_{\zeta,1}\parentheses{n\Gamma(\zeta)t^{\zeta}}.
\end{align*}
\end{lemma}

{\bf Acknowledgements.}  The authors would like to thank the anonymous referees for
their valuable comments leading to the improvement of our manuscript. This work was supported by Vietnam National Foundation for Science
and Technology Development (NAFOSTED) under grant number 101.02-2019.09. {The research  of T. C. has been partially supported by the Spanish Ministerio de Ciencia, Innovaci\'on y Universidades (MCIU), Agencia Estatal de Investigaci\'on (AEI) and Fondo Europeo de Desarrollo Regional (FEDER) under the project PGC2018-096540-B-I00, and by Junta de Andaluc\'{\i}a (Consejer\'{\i}a de Econom\'{\i}a y Conocimiento)   and FEDER under projects  US-1254251 and P18-FR-4509}.

\end{document}